\documentclass[hidelinks,onefignum,onetabnum]{siamart251216}
\usepackage{subcaption}
\usepackage{booktabs}
\usepackage{float}
\usepackage{pgfplots}
\pgfplotsset{compat=1.18}  
\usepgfplotslibrary{groupplots}

\usepackage{lipsum}
\usepackage{amsfonts}
\usepackage{graphicx}
\usepackage{epstopdf}
\usepackage{bm}
\usepackage{algorithmic}
\ifpdf
  \DeclareGraphicsExtensions{.eps,.pdf,.png,.jpg}
\else
  \DeclareGraphicsExtensions{.eps}
\fi

\newsiamremark{remark}{Remark}
\newsiamremark{hypothesis}{Hypothesis}
\crefname{hypothesis}{Hypothesis}{Hypotheses}
\newsiamthm{claim}{Claim}

\headers{Decorated particle method for Vlasov--Poisson}{M. B. Quashie, J. W. Burby, A. J. Christlieb, and Q. Tang}

\title{A Structure-Preserving Decorated Particle Method for the Vlasov--Poisson System\thanks{Submitted to the editors May 2026.
\funding{This work was partially supported by the ASCR program of Mathematical Multifaceted Integrated Capability Center (MMICC). QT was partially supported by the Department of Energy Office of Science, Early Career Research Program under Award Number DE-SC0026277.}}}

\author{Mandela B.~Quashie\thanks{Department of Computational Mathematics, Science and Engineering, Michigan State University, East Lansing, MI  (\email{mbq@msu.edu}, \email{christli@msu.edu}).}
\and J.~W.~Burby\thanks{Department of Physics and Institute for Fusion Studies, The University of Texas at Austin, Austin, TX (\email{joshua.burby@austin.utexas.edu}).}
\and Andrew J.~Christlieb\footnotemark[2]
\and Qi Tang\thanks{School of Computational Science and Engineering, Georgia Institute of Technology, Atlanta, GA (\email{qtang@gatech.edu}).}}

\usepackage{amsopn}

\usepackage[dvipsnames]{xcolor} 

\hypersetup{
  colorlinks=true,    
  linkcolor=blue, 
  citecolor=BrickRed, 
  urlcolor=teal       
}

\ifpdf
\hypersetup{
  pdftitle={Decorated particle method for Vlasov--Poisson},
  pdfauthor={M. B. Quashie, J. W. Burby, A. J. Christlieb, and Q. Tang}
}
\fi

\begin{document}

\maketitle

\begin{abstract}
We revisit the Scovel--Weinstein framework (\emph{Scovel \& Weinstein, CPAM 1994}) for reducing the Vlasov--Poisson system while preserving its Hamiltonian structure. Standard particle-in-cell (PIC) algorithms approximate the distribution function by macro-particles with position and velocity. In contrast, Scovel-Weinstein \emph{decorated} particles involve additional shape degrees of freedom, while maintaining a finite-dimensional reduction with Hamiltonian structure inherited from the continuum model. Although the original work established this structure three decades ago, its computational potential has remained largely unexplored. We present a practical implementation of the Scovel--Weinstein model and compare it with a standard PIC algorithm. Numerical experiments demonstrate that macro-particles in standard PIC can be replaced by far fewer decorated particles while retaining comparable accuracy. This decorated particle approach offers a new structure-preserving paradigm for kinetic plasma simulation.
\end{abstract}



\section{Introduction}
The Vlasov--Poisson (VP) system is a standard kinetic description of collisionless electrostatic plasmas, modeling the self-consistent evolution of charged particles under electrostatic forces. The VP system possesses a Hamiltonian structure of Lie--Poisson type  with exact conservation laws and an infinite family of Casimir invariants~\cite{morrison1980maxwell,morrison2017structure}. This geometric viewpoint, consistent with the broader interpretation of continuum dynamics as motion on diffeomorphism groups~\cite{arnold1966geometrie,holm1998euler,marsden2013introduction}, has inspired the development of modern structure-preserving plasma models and algorithms. See the textbooks \cite{Abraham_2008,marsden2013introduction,holm_dynamics_2011,hazeltine2018framework} for more details.
An important class of kinetic modeling algorithms contains the particle-in-cell (PIC) algorithms~\cite{birdsall2018plasma}, which represent the distribution function using finite collections of computational particles and the electric field using a spatial mesh. These methods notably avoid introducing a mesh in high-dimensional phase space.

Over the past decade, substantial progress has been made toward structure-preserving PIC algorithms.  
Early efforts~\cite{chen2011energy,chacon2013charge,chen2014energy,chen2015multi} introduced implicit schemes that conserve energy and charge, while later work~\cite{qin2015canonical,qiang2018symplectic} developed symplectic integrators for long-time Hamiltonian evolution of the Vlasov--Poisson and Vlasov--Maxwell systems. The remarkable paper \cite{crouseilles_hamiltonian_2015} introduced the idea of structure-preserving PIC based on Hamiltonian splitting. Although this work was flawed due to its use of a Poisson bracket for the Vlasov-Maxwell system that does not satisfy the Jacobi identity,
later work addressed the flaw, leading to the modern geometric electromagnetic PIC framework \cite{he_hamiltonian_2015,Kraus_GEMPIC_pub_2017} that features explicit time advance steps and exact preservation of Hamiltonian structure.  This framework inspired several finite-element and variational extensions~\cite{campos2022variational,wang2021geometric}. More recently, a generalized-momentum formulation enforcing the Lorenz gauge~\cite{christlieb2024particleII,christlieb2025particleI,christlieb2025particleIII} unified consistent particle-field coupling through a family of gauge-compatible time integrators, while \cite{ricketson2025explicit} demonstrated exact energy conservation in a fully-explicit PIC scheme for the first time. Recent reviews~\cite{morrison2017structure,ren2024recent} highlight these advances as part of a broader effort to integrate Hamiltonian structure directly into algorithmic design.

Standard PIC uses a Dirac delta function or a smooth B-spline as a non-dynamical shape function to approximate distribution functions, which introduces sampling noise and demands large particle counts for accurate statistics. To mitigate this, the Complex Particle Kinetic (CPK) method~\cite{hewett2003fragmentation} introduced Gaussian particles  with adaptive shape parameters. By allowing particles to capture local structure in phase space, CPK improved accuracy with fewer particles and motivated the idea that enriching particle structure could reduce cost without compromising fidelity. 
Another enriching idea was proposed by Scovel and Weinstein~\cite{scovel1994finite}, in which particles in PIC acquire shape degrees of freedom (DOFs) by allowing them to carry derivatives of Dirac distributions up to arbitrary order. In this work we refer to a particle together with its shape degrees of freedom as a \emph{decorated particle}. This construction yields finite-dimensional Lie--Poisson systems in which particle shapes evolve without artificial dissipation and preserve the geometric structure of the underlying VP dynamics. 
Since then, no systematic numerical development has appeared in the published literature that implements the full Scovel--Weinstein method. The closest anyone has come is P. Channell in the  paper \cite{channell1995canonical}, which can be understood as implementing a simplification of the Scovel-Weinstein method that assumes vanishing phase space moments of order $0$ and $1$. Notably, the only recent work related to the Scovel--Weinstein technique is a Hamiltonian fluid closure framework~\cite{burby2023variable}. Consequently, the Scovel--Weinstein PIC (SWPIC) framework has remained largely unexplored, despite offering a novel pathway for model reduction. Rather than truncating moment hierarchies or filtering kinetic information, SWPIC embeds phase-space structure into a nonstandard particle representation. 

The present work develops the first numerical implementation of the original Scovel--Weinstein method and compares it with traditional methods. Numerical experiments indicate that the use of decorated particles allows a substantial reduction in particle count relative to standard PIC while maintaining comparable accuracy in field
evolution and bulk kinetic quantities. This reduction is achieved without noticeable degradation in growth rates, damping behavior, or energy evolution, and is accompanied by a decrease in statistical noise. We consider only the first non-trivial extension of PIC offered by SWPIC, which attaches phase space moments of orders $0$ (Dirac) and $1$ (first derivative of Dirac) to each particle. We will explore implementations with moment orders from $0$ to $N>1$ in future work. 

The remainder of this paper is organized as follows. Section~\ref{sec:swmodel} introduces the Scovel--Weinstein reduction and the resulting SWPIC formulation. Sections~\ref{sec:init}--\ref{sec:discrete-hamiltonian} describe initialization, numerical analysis, and the discrete Hamiltonian structure. Section~\ref{sec:numerical} presents numerical experiments, including test-particle dynamics, two-stream instability, and strong Landau damping. Conclusions are given in Section~\ref{sec:conclusion}.

\section{Vlasov--Poisson System}
We consider the VP system in $d$ spatial dimensions on a domain $\mathcal{Q}$. We allow the space $\mathcal{Q}$ to be either unbounded, $\mathcal{Q} = \mathbb{R}^d$, or periodic, $\mathcal{Q} = \mathbb{T}^d$, where $\mathbb{T} = \mathbb{R}\text{ mod }L$ and $L$ is a box side length. 
The unknowns are the distribution function
\(
f:\mathcal{Q}\times\mathbb{R}^d\times\mathbb{R}\to\mathbb{R},
\)
and the electrostatic potential
\(
\varphi:\mathcal{Q}\times\mathbb{R}\to\mathbb{R}.
\)

The VP system is given by 
\begin{gather}
\partial_t f + \frac{p}{m}\cdot\nabla_q f - e\,\nabla_q \varphi \cdot \nabla_p f = 0, \label{eq:vlasovd} \\[0.5em]
-\epsilon_0\,\Delta \varphi = e\left(\int f(q,p)\,dp - n_0(q)\right) \label{eq:poissond},
\end{gather} where $f(q,p)$ denotes the particle distribution function and $\varphi(q)$ the electrostatic potential, with time dependence understood. Here we use subscripts to denote partial gradients of functions of $(q,p)$.  The function $n_0(q)$ represents a background of motionless particles with charge $-e$ that ensures ensure overall charge neutrality. The right-hand side of \eqref{eq:poissond} therefore has zero spatial mean. Accordingly, $\varphi$ is determined up to an additive constant that we fix by requiring the spatial average of $\varphi$ vanishes. Particles have mass $m$, charge $e$, position $q\in\mathcal{Q}$, and momentum $p\in\mathbb{R}^d$. 

The VP system admits a Hamiltonian formulation on the space of distribution
functions $f$ that we summarize below. The Hamiltonian functional is
\begin{equation}
\label{eq:vp_hamiltonian}
H(f)
= \int
\frac{|p|^2}{2m}\,f(q,p)\,dq\,dp
+ \frac{\epsilon_0}{2}
\int |\nabla \varphi(q)|^2\,dq,
\end{equation}
where the electrostatic potential $\varphi$ is regarded as a functional of $f$ through the Poisson equation~\eqref{eq:poissond}. The first term represents the kinetic energy of the particle distribution, while the second term represents the electrostatic field energy. The Hamiltonian structure of the VP system \cite{morrison2017structure} is a natural starting point for structure-preserving plasma algorithms. 
The starting point for this work is the decorated particle formulation in~\cite{scovel1994finite}, which will be discussed carefully below.

\section{Scovel--Weinstein phase space moment model}
\label{sec:swmodel} In the Scovel--Weinstein formulation, standard marker particles from PIC are replaced with decorated particles. Each decorated particle carries moment degrees of freedom based on a local expansion of the distribution function that captures deformations in phase space. If only the lowest-order moments (order $0$) are included, decorated particles reduce to marker particles. Their equations of motion comprise a finite-dimensional noncanonical Hamiltonian system, ensuring moment evolution without artificial dissipation. Here we summarize the key
results of~\cite{scovel1994finite} and the theoretical ingredients needed for numerical implementation. First we will assume that the spatial domain is unbounded $\mathcal{Q} = \mathbb{R}^d$. We indicate how the model changes when $\mathcal{Q} = \mathbb{T}^d$ in Section~\ref{the_sw_model}.

\subsection{Hamiltonian perspective on PIC}
\label{sec:frompic}
In analytical descriptions of PIC methods, the distribution function is written in Klimontovich form~\cite{klimontovich2013statistical} as
\begin{equation}
f(q,p)
  = \sum_{A=1}^{N_{\text{PIC}}} \psi^*_A\,\delta(q - Q_A)\,\delta(p - P_A),
\label{eq:fpic}
\end{equation}
where each marker particle \(\zeta_A=(Q_A,P_A,\psi^*_A)\) carries a position and momentum $Z_A =(Q_A,P_A)$, as well as a weight $\psi^*_A$.   The delta functions localize each marker particle in phase space. The contribution to the Poisson equation from particle-$A$ is the charge density \(e\,\psi^*_A\,\delta(q-Q_A)\). The spatial $\delta$-function in Poisson's equation is sometimes replaced by a compactly supported shape function \(S_h(q-Q_A)\), which regularizes charge deposition while recovering~\eqref{eq:fpic} in the limit \(S_h\to\delta\).  
The representation \eqref{eq:fpic} is coupled to a spatial discretization used to compute the self-consistent field. Particle charge is deposited onto a spatial mesh, the electrostatic potential is solved numerically, and the resulting field is evaluated at particle locations when pushing particles. For the present discussion, these steps are viewed abstractly as operators mapping particle data to fields, without specifying a particular discretization.

The continuous-time dynamical model that underlies standard PIC can be derived from the Hamtiltonian formulation of the VP system as follows. Let
\[
  \mathfrak g = C^\infty(\mathbb{R}^{2d})
\]
denote the Lie algebra of smooth phase-space observables, equipped with the canonical Poisson bracket
\begin{equation}
  \{h_1,h_2\}_{\mathrm{can}}
   = \nabla_q h_1\cdot \nabla_p h_2 - \nabla_qh_2\cdot\nabla_p h_1.
\label{eq:canonicalbracket}
\end{equation}
The dual space \(\mathfrak g^*\) may be identified with the space of phase-space distribution functions $f$. For any observable \(h\in\mathfrak g\) and any distribution \(f\in\mathfrak g^*\), the pairing is
\begin{equation}
  \langle f,h\rangle
    = \int f(q,p)\,h(q,p)\,dq\,dp.
\label{eq:pairing}
\end{equation}
The Lie--Poisson construction
\cite{morrison1980maxwell,morrison1986paradigm,marsden1982hamiltonian}
then uses the Lie algebra structure on $\mathfrak{g}$ to define a Poisson bracket between functionals \(F,G:\mathfrak g^*\to\mathbb{R}\) given by
\begin{equation}
  \{F,G\}(f)
   = \int f\,
      \Bigl\{
         \frac{\delta F}{\delta f},
         \frac{\delta G}{\delta f}
      \Bigr\}_{\mathrm{can}}
      dq\,dp,
\label{eq:lp-vp}
\end{equation}
where $\delta F/\delta f$ denotes the functional derivative of $F$ with respect to $f$ defined by
\begin{align*}
    \frac{d}{d\epsilon}\bigg|_{\epsilon=0}F(f+\epsilon\,\delta f) = \int \frac{\delta F}{\delta f}\delta f\,dq\,dp,\quad \forall \delta f\in\mathfrak{g}^*.
\end{align*}
With Hamiltonian \eqref{eq:vp_hamiltonian}, 
Hamilton’s equation
\(
  \partial_t F = \{F, H\}
\)
for arbitrary $F:\mathfrak{g}\rightarrow \mathbb{R}$ recovers the VP system in weak form. 

Let $\bm{\zeta} = (\zeta_1,\dots,\zeta_{N_{\text{PIC}}})$ denote the multi-marker-particle state. The Klimontovich representation~\eqref{eq:fpic} defines a Poisson map 
\[
\Gamma_0 : \bm{\zeta}
   \longmapsto
   \sum_{A=1}^{N_{\text{PIC}}} \psi^*_A\,\delta(q-Q_A)\,\delta(p-P_A),
\]
from the weighted multi-particle phase space $(\mathbb{R}^d\times\mathbb{R}^{d}\times\mathbb{R})^{N_{\text{PIC}}}$ to \(\mathfrak g^*\). The Poisson bracket between functions $\mathcal{F}(\bm{\zeta}),\mathcal{G}(\bm{\zeta})$ on the weighted multi-particle phase space is
\begin{equation}
\{\mathcal{F},\mathcal{G}\}_{\mathrm{PIC}}
 =
 \sum_{A=1}^{N_{\text{PIC}}} \frac{1}{\psi^*_A}
 \left(\nabla_{Q_A}\mathcal{F}\cdot \nabla_{P_A}\mathcal{G} -\nabla_{Q_A}\mathcal{G}\cdot \nabla_{P_A}\mathcal{F}
 \right),
\label{eq:pic_bracket}
\end{equation}
while $\mathfrak{g}^*$ is equipped with the Lie-Poisson bracket \eqref{eq:lp-vp}.  Note that each of the weights $\psi^*_A$ is a Casimir for $\{\cdot,\cdot\}_{\text{PIC}}$. Composing the continuum Hamiltonian \eqref{eq:vp_hamiltonian} with \(\Gamma_0\) produces the discrete particle Hamiltonian
\begin{equation}
H_{\text{PIC}}(\bm{\zeta}) = \sum_{I=1}^{N_{\text{PIC}}} \psi^*_A\,\frac{|P_A|^2}{2m}
  + \frac{\epsilon_0}{2}
    \int |\nabla\varphi(q)|^2 dq,
\label{eq:pic_hamiltonian}
\end{equation}
where the potential is regarded as a function of $\bm{\zeta}$ as follows. Let $\mathcal{E}\subset H^1(\mathbb{R}^d)$ denote some subspace, which represents the chosen method of discretizing the electrostatic potential, such as an $H^1$ finite element basis.
The potential $\varphi$ is then defined as an element of $\mathcal{E}$ such that
\begin{align}
    \forall v\in\mathcal{E},\quad \epsilon_0\left\langle \nabla\varphi,\nabla v\right\rangle = \left\langle  e\bigg(\sum_{A=1}^{N_{\text{PIC}}} \psi^*_A\,S_h(q-Q_A)-n_0\bigg),v\right\rangle\label{eq:pic_poisson},
\end{align}
where $\langle\cdot,\cdot\rangle$ denotes the usual $L^2$ integration pairing on $\mathbb{R}^d$. When $\mathcal{E}=H^1(\mathbb{R}^d)$ and $S_h = \delta$ this definition of $\varphi$ leads to the usual Coulomb formula for the potential produced by each particle. Otherwise it effectively smooths the Coulomb potential, both by projection onto the subspace $\mathcal{E}$ and by mollifying the singular charge density using $S_h$. Many PIC implementations choose $S_h = \delta$. 
Hamilton’s equations associated with the bracket~\eqref{eq:pic_bracket} and Hamiltonian~\eqref{eq:pic_hamiltonian} reduce to
\begin{equation}
\dot{Q}_A = \frac{P_A}{m},
\qquad
\dot{P}_A = -\,e\,\nabla\varphi(Q_A),
\label{eq:pic_dynamics}
\end{equation}
which resemble the standard electrostatic $N$-body problem equations of motion, but involve the smoothed potential just described. By Guillemin--Sternberg collectivization \cite{guillemin1980moment}, if $\bm{\zeta}(t)$ is a solution of the Hamiltonian system defined by $\{\cdot,\cdot\}_{\text{PIC}}$ and $H_{\text{PIC}}$ then its image under the Poisson map $\Gamma_0$ is a (singular) exact solution of a modified VP system, where the electrostatic potential is defined according to 
\begin{align}
    \forall v\in\mathcal{E},\quad\epsilon_0\langle \nabla \varphi,\nabla v\rangle = \langle \rho,v\rangle,\quad \rho(q) = e\int \bigg(f(\overline{q},p)S_h(q-\overline{q})\,d\overline{q}\bigg)\,dp - e n_0(q).\label{eq:discrete_poisson}
\end{align}
Note that the shape function $S_h$ need not tend to a Dirac-$\delta$ in the limit $N_{\text{PIC}}\rightarrow \infty$.

\subsection{Scovel--Weinstein enhancement of PIC}
Although the Klimontovich-PIC ansatz leads to a practical simulation algorithm, it
contains no information about local phase-space gradients. Scovel--Weinstein overcome this limitation by equipping each marker particle with additional moment variables, enriching the PIC representation while respecting the Lie-Poisson structure of the VP system.

The Scovel--Weinstein construction begins with the decomposition
\[
\mathfrak g=\mathfrak b+\mathfrak s,
\]
and its basic Lie-theoretic consequences, which we now describe. Here $\mathfrak b$ contains the affine functions of $z=(q,p)$, $\mathfrak s$ contains functions with Taylor expansions about the origin that begin at quadratic order, and $+$ denotes the vector space direct sum. Both $\mathfrak{b}$ and $\mathfrak{s}$ are Lie subalgebras of $\mathfrak{g}$. Notably, $\mathfrak{b}$ is a finite-dimensional Lie algebra whose underlying Lie group is the famous Heisenberg group $\mathcal{B} = \mathbb{R}^d_Q\times\mathbb{R}^d_P\times\mathbb{R}_\Psi$ with group product
\[
(Q_1,P_1,\Psi_1)(Q_2,P_2,\Psi_2)
  =
\bigl(
Q_1+Q_2,\,
P_1+P_2,\,
\Psi_1+\Psi_2+\tfrac12(Q_1\cdot P_2-Q_2\cdot P_1)
\bigr).
\]
Here $(Q,P)$ represents a phase space translation, while $\Psi$ represents a classically-unobservable phase. The Lie algebra isomorphism relating the Heisenberg algebra to affine phase space functions is $(\delta q,\delta p,\delta \psi)\mapsto f_{(\delta q,\delta p,\delta\psi)} $, where $f_{(\delta q,\delta p,\delta \psi)}(q,p) =  \delta\psi - \delta p\cdot q + \delta q\cdot p$.

For each non-negative integer $k$ let $\mathcal I_k$ be the elements of $\mathfrak g$ with Taylor expansions about the origin that begin at order $k$. For $k\geq 2$ the space $\mathcal{I}_k\subset \mathfrak{s}$ is an ideal inside of $\mathfrak{s}$. (It is \emph{not} an ideal in $\mathfrak{g}$!) The vector space quotient
$\mathfrak s_k=\mathfrak s/\mathcal I_k$, $k\geq 2$,
which may be understood as phase space polynomials with degree greater than $1$ and less than $k$, therefore inherits a Lie algebra structure. (Our convention is that the zero function is the only phase space polynomial with degree greater than $1$ and less than $2$.) 
The quotient map
$
\pi_k:\mathfrak s\rightarrow\mathfrak s_k,
$
assigns to each function $s\in\mathfrak{s}$ its degree-$(k-1)$ Taylor polynomial at the origin. 
Its dual
$
\pi_k^*:\mathfrak s_k^*\rightarrow\mathfrak s^*\subset\mathfrak{g}^*,
$
realizes vectors of centered moments with degree greater than $1$ and less than $k$ as singular distribution functions concentrated at the origin.  More precisely, if $s^*_k = (M_2,\dots, M_{k-1})$ denotes a vector of centered moments in $\mathfrak{s}_k^*$ then $\pi_k^*(s^*_k) =\sum_{\ell=2}^{k-1}\frac{(-1)^\ell}{\ell!}  M_\ell^{\mu_1\dots \mu_\ell}\partial^\ell_{\mu_1\dots \mu_\ell}\delta$ is the corresponding singular distribution function. Here we reserve the symbol $\mu\in\{1,\dots,2d\}$ for indexing phase space coordinates and adopt the usual summation convention for repeated instances of $\mu$. It is straightforward to show that $\pi_k$ is a Lie algebra homomorphism and that $\pi_k^*$ is a Poisson map between Lie-Poisson spaces.

Next Scovel--Weinstein introduce the most novel technical element in their formalism: a mapping $\Gamma : (\mathcal{B}\times\mathfrak{b}^*)\times \mathfrak{s}^*\rightarrow\mathfrak{g}^*$ given by $\Gamma(B,b^*,s^*) = \text{Ad}^*_B(b^*+s^*)$, where $\text{Ad}^*$ denotes the (left) coadjoint action of $\mathcal{G}$, the infinite-dimensional Lie group underlying $\mathfrak{g}$, on $\mathfrak{g}^*$. For the VP system, the various definitions underlying this abstract expression unpack as follows. 

The variable $B=(Z,\Psi)\in \mathbb{R}^{2d}\times\mathbb{R}$ encodes the usual PIC single-particle state $Z = (Q,P)$ together with an unobservable phase $\Psi$. The dual variable $b^* = (z^*,\psi^*)\in\mathbb{R}^{2d}\times\mathbb{R}$, with $z^*=(q^*,p^*)$, encodes the centered phase space moments of degree $0$ and $1$, while $s^*$ denotes the part of the distribution function with vanishing centered moments of order $0$ and $1$. The value of $\Gamma(b,b^*,s^*)$ is given explicitly by
\begin{align*}
    \Gamma(b,b^*,s^*)(z) = \psi^*\delta(z-Z) +p^*\cdot \nabla_{q}\delta(z-Z) - q^{*}\cdot\nabla_{p}\delta(z-Z)+ s^*(z-Z).
\end{align*}
The domain space for $\Gamma$, $(\mathcal{B}\times\mathfrak{b}^*)\times \mathfrak{s}^*$, has a natural product Poisson manifold structure. The Poisson bracket on the factor $\mathcal{B}\times\mathfrak{b}^*$ follows from applying left trivialization to the cotangent bundle $T^*\mathcal{B}$, with its canonical Poisson bracket. The factor $\mathfrak{s}^*$ is equipped with its Lie-Poisson bracket. 

Scovel--Weinstein prove the remarkable fact that, relative to the indicated Poisson manifold structures, $\Gamma$ is a Poisson map. This is a general Lie-theoretic result, and therefore enjoys applicability outside the context of phase space moments in the VP system. For example it serves as the basis for the Hamiltonian fluid moment closures described in \cite{burby2023variable}.  

Finally, to completely specify their decorated particle model, Scovel--Weinstein synthesize the Poisson maps $\pi_k^*$ and $\Gamma$ to build a Poisson map from the $N_p$ decorated particle phase space $((\mathcal{B}\times\mathfrak{b}^*)\times \mathfrak{s}_k^*)^{N_p}\ni\tilde{\bm{\zeta}} = (\tilde{\zeta}_1,\dots,\tilde{\zeta}_{N_p})$, with its evident product Poisson manifold structure, into $\mathfrak{g}^*$. This map, which we denote $\bm{\Gamma}_k$, is given in terms of the single-decorated-particle Poisson maps
\begin{align*}
    \Gamma_k:(\mathcal{B}\times\mathfrak{b}^*)\times \mathfrak{s}_k^*\xrightarrow{\text{id}_{\mathcal{B}\times\mathfrak{b}^*}\times \pi_k^*} (\mathcal{B}\times\mathfrak{b}^*)\times \mathfrak{s}^*\xrightarrow{\Gamma}\mathfrak{g}^*,
\end{align*}
as 
\begin{align}
    \bm{\Gamma}_k(\tilde{\bm{\zeta}}) = \sum_{a=1}^{N_p}\Gamma_k(\tilde{\zeta}_a).\label{multi_particle_sw_poisson_map}
\end{align}

The single-decorated-particle phase space $(\mathcal{B}\times\mathfrak{b}^*)\times \mathfrak{s}_k^*$ comprises tuples of the form $\widetilde{\zeta} =(Z,\Psi,z^*,\psi^*,M_2,\dots,M_{k-1})$. The Hamiltonian on $N_p$-decorated-particle phase space space is 
\[
H_{\text{SWPIC}}(\tilde{\bm{\zeta}}) = H(\bm{\Gamma}_k(\tilde{\bm{\zeta}})),
\]
where $H$ is defined in \eqref{eq:vp_hamiltonian}. The equations of motion for the decorated particles are Hamilton's equations with Hamiltonian $H_{\text{SWPIC}}$ and the Poisson bracket on $((\mathcal{B}\times\mathfrak{b}^*)\times \mathfrak{s}_k^*)^{N_p}$ described above.

As in the discussion of standard PIC above, Guillemin--Sternberg collectivization implies that any solution $\widetilde{\bm{\zeta}}(t)$ of Hamilton's equations on $(\mathcal{B}\times\mathfrak{b}^*)\times \mathfrak{s}_k^*$ maps along $\bm{\Gamma}_k$ onto an exact solution of the modified VP system where the electrostatic potential is computed using \eqref{eq:discrete_poisson}. 

Finding the equations of motion for each decorated particle in the Scovel-Weinstein model requires explicit expressions for the Poisson map $\bm{\Gamma}_k$ and the Poisson bracket on $((\mathcal{B}\times\mathfrak{b}^*)\times \mathfrak{s}_k^*)^{N_p}$. The mapping $\bm{\Gamma}_k$ should be understood as specifying an ansatz for the distribution function generalizing the Klimontovich ansatz \eqref{eq:fpic}. These expressions are given as follows.

According to Eq.\,\eqref{multi_particle_sw_poisson_map} $\bm{\Gamma}_k$ is a sum of single-decorated-particle Poisson maps. The single-decorated-particle Poisson map $\Gamma_k$ is given explicitly by
\begin{align}
    \Gamma_k(\widetilde{\zeta}) &=\psi^*\,\delta_Z
  + p^*\!\cdot\nabla_q\delta_Z
  - q^*\!\cdot\nabla_p\delta_Z +\sum_{\ell=2}^{k-1}\frac{(-1)^\ell}{\ell!}
   (\mathbb M)^{i\mu_1\cdots \mu_\ell}\,
   \partial_{\mu_1\cdots \mu_\ell}^{\ell}\delta_Z,\label{sw_f_ansatz}
\end{align}
where $\delta_Z(z) := \delta(z-Z)$.
It recovers the usual PIC ansatz when $z^* = 0$, a monopole–dipole model at $k=2$, and higher-moment models for $k > 2$. The numerical results presented in this work apply to the monopole-dipole model only.

The Poisson bracket on $((\mathcal{B}\times\mathfrak{b}^*)\times \mathfrak{s}_k^*)^{N_p}$ is given by summing single-decorated-particle Poisson brackets on the single-decorated-particle phase spaces $(\mathcal{B}\times\mathfrak{b}^*)\times \mathfrak{s}_k^*$. Each single-decorated-particle bracket is given by $\{f,g\} = \{f,g\}_{\mathcal{B}\times\mathfrak{b}^*} + \{f,g\}_{\frak{s}_k^*}$. The bracket $\{f,g\}_{\mathcal{B}\times\mathfrak{b}^*}$ follows from applying left-trivialization to the canonical Poisson bracket on $T^*\mathcal{B}$, as described in \autoref{app:lefttrivialization} and \autoref{app:poissonbracket}.
The result is
\begin{align*}
\{f,g\}_{\mathcal{B}\times\mathfrak b^*}
&= \left(
\frac{\partial f}{\partial Z^\mu}\frac{\partial g}{\partial z_\mu^*}
-
\frac{\partial g}{\partial Z^\mu}\frac{\partial f}{\partial z_\mu^*}
\right)
+ \frac{\partial f}{\partial\Psi}\frac{\partial g}{\partial\psi^*}
- \frac{\partial g}{\partial\Psi}\frac{\partial f}{\partial\psi^*}\\
&\quad -\psi^*\left(
\frac{\partial f}{\partial q_i^*}\frac{\partial g}{\partial p^{*i}}
-
\frac{\partial g}{\partial q_i^*}\frac{\partial f}{\partial p^{*i}}
\right)\\
&\quad + \frac{1}{2}\!\left[
\frac{\partial f}{\partial\Psi}
\!\left(
Q_i\frac{\partial g}{\partial p^{*i}}
-
P_i\frac{\partial g}{\partial q^{*i}}
\right)
-
\frac{\partial g}{\partial\Psi}
\!\left(
Q_i\frac{\partial f}{\partial p^{*i}}
-
P_i\frac{\partial f}{\partial q^{*i}}
\right)
\right].
\end{align*}
Here we use the symbols $\mu\in\{1,\dots,2d\}$ and $i\in\{1,\dots, d\}$ for indexing components and adopt the Einstein summation convention. The bracket $\{f,g\}_{\frak{s}_k^*}$ is given by the Lie-Poisson formula
\(
\{f,g\}_{\frak{s}_k^*} = \langle s^*,[\delta f/\delta s^*,\delta g/\delta s^*]_k\rangle,
\)
where $[\cdot,\cdot]_k$ denotes the Lie bracket on the quotient Lie algebra $\mathfrak{s}_k$. When $k=2$ this bracket vanishes because $\mathfrak{s}_2 = \{0\}$. We do not display the explicit formulas for $k>2$ because the numerical results in the paper apply to $k=2$ only.

Notice that the distribution function $\Gamma_k(\widetilde{\zeta})$ in Eq.\,\eqref{sw_f_ansatz} does not depend on the phase variable $\Psi$. This implies that the Hamiltonian $H_{\text{SWPIC}}$ will not depend on any of the phase variables $\Psi_a$. The corresponding conserved ``momenta" are the $\psi_a^*$. As noted previously by Scovel--Weinstein, we are therefore justified in restricting attention to the reduced phase space for each particle,
\[\mathcal{P}_k = (\mathbb{R}^{d}\times\mathbb{R}^d\times \mathbb{R}^{d}\times\mathbb{R}^d\times\mathbb{R})\times \mathfrak{s}_k^*\ni (Q,P,q^*,p^*,\psi^*,M_2,\dots, M_{k-1})=\widetilde{\zeta}.\]
The Poisson bracket on $\mathcal{P}_k$ is given by replacing $\{\cdot,\cdot\}_{\mathcal{B}\times\mathfrak{b}^*}$ with
\begin{align}
\{f,g\}_{2}
&= \left(
\frac{\partial f}{\partial Z^\mu}\frac{\partial g}{\partial z_\mu^*}
-
\frac{\partial g}{\partial Z^\mu}\frac{\partial f}{\partial z_\mu^*}
\right)
- \psi^*\!\left(
\frac{\partial f}{\partial q_i^*}\frac{\partial g}{\partial p^{*i}}
-
\frac{\partial g}{\partial q_i^*}\frac{\partial f}{\partial p^{*i}}
\right),
\label{eq:sw_bracket}
\end{align}
while keeping the Lie--Poisson bracket on $\mathfrak{s}_k^*$ unchanged. The phase variable $\Psi$ will not appear in any of the analysis that follows. This justifies our reusing the symbol $\widetilde{\zeta}$ for points in the reduced phase space $\mathcal{P}_k$.

\subsection{Example dynamics of a single decorated particle}
\label{sec:singlelump-dynamics}

To make the structure of the ($k=2$) monopole--dipole Scovel--Weinstein model transparent, we temporarily restrict to the case with a single decorated particle ($N_p=1$) and zero electric charge $e = 0$. For illustrative purposes, we also introduce a prescribed potential $V(q)$. The Hamiltonian for the Vlasov equation is then
\begin{align*}
    H(f) = \int \bigg(\frac{|p|^2}{2m} + V(q)\bigg)f\,dq\,dp.
\end{align*}
Since $N_p=1$ the Scovel--Weinstein Hamiltonian is $H_{\text{SWPIC}} = H\circ \Gamma_2$.
Unpacking definitions, we find
\begin{align*}
    H_{\text{SWPIC}}(\widetilde{\zeta}) = &\int \bigg(\frac{|p|^2}{2m} + V(q)\bigg)\bigg(\psi^*\delta(q-Q)\delta(p-P)\nonumber\\
    &+p^*\cdot \nabla_q\delta(q-Q)\,\delta(p-P)-q^*\cdot\nabla_p\delta(p-P)\,\delta(q-Q)\bigg)\,dq\,dp\nonumber\\
    =&\psi^*\bigg(\frac{|P|^2}{2m} + V(Q)\bigg) - p^*\cdot\nabla V(Q) + \frac{1}{m}q^*\cdot P.
\end{align*}
Using this Hamiltonian and the Poisson bracket on $\mathcal{P}_2$ defined in~\eqref{eq:sw_bracket}$,$ the evolution equations for $\widetilde{\zeta} =(Q,P,q^*,p^*,\psi^*)\in \mathcal{P}_2$ become
\begin{subequations}
\label{eq:sw_prescribed_potential}
\begin{align}
\dot{Q} &= \frac{P}{m},  
&\dot{P} &= -\nabla V(Q),  & \dot{\psi}^* &= 0, \\
\dot{p}^* &= -\frac{q^*}{m},  
&\dot{q}^* &= p^*\!\cdot\nabla\nabla V(Q).
\end{align}
\end{subequations}
The variables $(Q,P)$ coincide with the phase space location at which the distribution function concentrates. They evolve according to the usual Hamilton equations. The dual variables $(q^*,p^*)$ obey the linearization of Hamilton's equations along the $(Q,P)$ trajectory. The linearized equations become more apparent after introducing the suggestive notation $p^* = \delta Q$ and $q^* = -\delta P$. (We will however not adopt this notation in what follows.) Each decorated particle with $k=2$ therefore captures the motion of a Vlasov characteristic and the infinitesimal deformation of nearby characteristics. This clarifies the sense in which the Scovel--Weinstein model incorporates gradient information when $k=2$.

\subsection{Explicit formulas for the Scovel--Weinstein model\label{the_sw_model}}
Again restrict to $k=2$. Restore the electric charge $e$ and set the external potential equal to zero $V=0$. The phase space for $N_p$ decorated particles is $\mathcal{P}_2^{N_p}$. The Hamiltonian is
\begin{align}
    H_{\text{SWPIC}}(\bm{\widetilde{\zeta}}) = \sum_{a=1}^{N_p}\left(\frac{\psi_a^*}{2m}|P_a|^2 + \frac{1}{m}q_a^*\cdot P_a\right) + \frac12\epsilon_0\,\int |\nabla\varphi(q)|^2\,dq,\label{eq:sw_hamiltonian}
\end{align}
where the electrostatic potential $\varphi$ is regarded as a functional of $\widetilde{\bm{\zeta}}$ as follows. Let $\mathcal{E}\subset H^1(\mathbb{R}^d)$ denote some subspace. 
The potential $\varphi$ is then defined as an element of $\mathcal{E}$ such that
\begin{align}
    \forall v\in\mathcal{E},\quad \epsilon_0\left\langle \nabla\varphi,\nabla v\right\rangle &= \left\langle  e\bigg(\sum_{a=1}^{N_{p}} \psi^*_a\,S_h(q-Q_a)-n_0\bigg),v\right\rangle\nonumber\\
    &+\left\langle  e\sum_{a=1}^{N_{p}} p^*_a\cdot\nabla S_h(q-Q_a),v\right\rangle\label{eq:swpic_poisson},
\end{align}
where $\langle\cdot,\cdot\rangle$ denotes the usual $L^2$ integration pairing on $\mathbb{R}^d$. When $p^*_a=0$ for all $a$ this recovers the definition of the potential in traditional PIC. Otherwise it adds a dipole correction to the charge density, consistent with a picture of decorated particles as pairs of closely-spaced marker particles. 

Hamilton’s equations associated with the bracket~\eqref{eq:sw_bracket} and Hamiltonian~\eqref{eq:sw_hamiltonian} reduce to
\begin{subequations}
\label{eq:swpic_dynamics}
\begin{align}
\dot{Q}_a &= \frac{P_a}{m},
&\dot{P}_a &= -e\int S_h(q-Q_a)\nabla\varphi(q)\,dq, & \dot{\psi}_a^* & = 0,\\
\dot{p}_a^* &= -\frac{q_a^*}{m},\qquad & \dot{q}_a^* &= -e\int p_a^*\cdot\nabla S_h(q-Q_a)\,\nabla\varphi(q)\,dq.
\end{align}
\end{subequations}
For each $a\in \{1,\dots, N_p\}$, the variables $(Q_a,P_a)$ evolve according to the usual Hamilton equations. The dual variables $(q^*_a,p^*_a)$ obey the linearization of Hamilton's equations along the $(Q_a,P_a)$ trajectory. By Guillemin--Sternberg collectivization, if $\widetilde{\bm{\zeta}}(t)$ is a solution of the Hamiltonian system defined by $\{\cdot,\cdot\}_{2}$ and $H_{\text{SWPIC}}$ then its image under the Poisson map $\bm{\Gamma}_2$ is a (singular) exact solution of the modified VP system, where \eqref{eq:discrete_poisson} is used to define the potential instead of \eqref{eq:poissond}.

So far we have described the Scovel--Weinstein model equations for an unbounded spatial domain $\mathcal{Q}=\mathbb{R}^d$. When $\mathcal{Q} = \mathbb{T}^d$ very little changes. The single-decorated-particle phase space becomes
\begin{align*}
    \mathcal{P}_k = (\mathbb{T}^d\times\mathbb{R}^d\times\mathbb{R}^d\times\mathbb{R}^d\times\mathbb{R})\times\mathfrak{s}_k^*\ni (Q,P,q^*,p^*,\psi^*,M_2,\dots, M_{k-1})=\widetilde{\zeta}.
\end{align*}
The Poisson bracket \eqref{eq:sw_bracket} is still a valid Poisson bracket when $Q\in\mathbb{T}^d$ instead of $\mathbb{R}^d$ because locally these two spaces are indistinguishable. The Hamiltonian is still given by \eqref{eq:sw_hamiltonian}. The potential is still computed using \eqref{eq:swpic_poisson}, but now $\mathcal{E}\subset H^1(\mathbb{T}^d)$ instead of $H^1(\mathbb{R}^d)$. The decorated particle equations of motion are still \eqref{eq:swpic_dynamics}. Although the derivation of the Scovel--Weinstein model used the vector space structure of $\mathcal{Q} = \mathbb{R}^d$ in an essential way at intermediate steps, after eliminating the unobservable phase $\Psi$ the spaces $\mathcal{Q} = \mathbb{R}^d$ and $\mathcal{Q} = \mathbb{T}^d$ become essentially interchangable.

\section{Initialization via Particle Moment Expansion}
\label{sec:init}
Having established the Scovel--Weinstein model, we now describe how to initialize its variables starting from a large ensemble of $N_{\text{PIC}}$ marker particles $\{(\overline{Q}_A,\overline{P}_A,\overline{\psi}_A^*)\}_{A=1,\dots,N_{\text{PIC}}}$.  
The phase space for $N_p$ decorated particles with $k=2$ is \[\mathcal{P}_2^{N_p}\ni (Q_1,P_1,q^*_1,p^*_1,\psi^*_1,\dots,Q_{N_p},P_{N_p},q^*_{N_p},p^*_{N_p},\psi^*_{N_p}).\]
An initial state for the Scovel-Weinstein model requires specifying the decorated particle variables
\[
(Q_a, P_a,  q_a^*, p_a^*,\psi_a^*) ,
\qquad a=1,\dots,N_p.
\]
We aim to compress the marker particle variables into a smaller collection of decorated particles that preserves distributional fidelity while reducing computational cost.  

Our compression strategy depends on whether $\mathcal{Q} = \mathbb{R}^d$ or $\mathcal{Q} = \mathbb{T}^d$. We will describe the method for $\mathcal{Q} = \mathbb{R}^d$ first, since it is more intuitive. Then we will describe the method for $\mathcal{Q} = \mathbb{T}^d$, which is slightly more technical.

\subsection{Initialization with $\mathcal{Q} = \mathbb{R}^d$}
To construct the reduced initialization, we partition the marker particle ensemble into 
$N_p \ll N_{\text{PIC}}$ clusters $S_{a}$ by applying the $k$-means algorithm to the phase-space coordinates $(\overline{Q}_A,\overline{P}_A)$.
Each cluster represents a localized aggregate of marker particles with total weight
\begin{equation}
\label{eq:psi-star}
\psi^*_a = \sum_{A \in S_a} \overline{\psi}^*_A .
\end{equation}
Here $S_a$ denotes the index set of marker particles belonging to cluster $S_a$ and $\overline{\psi}^*_A$ denote the marker particle weights. We find $A_0\in S_a$ corresponding to the particle $(\overline{Q}_{A_0},\overline{P}_{A_0})$ closest to the true mean phase space location of the cluster
\begin{align*}
\overline{Z}_a = \frac{\sum_{A\in S_a}\overline{\psi}^*_A\overline{Z}_A}{\sum_{A\in S_a}\overline{\psi}^*_A}.
\end{align*}
We then define $(Q_a,P_a) = (\overline{Q}_{A_0},\overline{P}_{A_0})$.  This coarse-graining step preserves the large-scale geometry of the underlying distribution.

To encode derivative information of the distribution in the decorated particles, we approximate the empirical distribution of each cluster using a first-order Taylor expansion about $(Q_a,P_a)$. This leads to the approximation for the cluster's empirical marker particle distribution given by
\begin{align*}
    \sum_{A\in S_a}\overline{\psi}_A^*\delta(q-\overline{Q}_A)\delta(p-\overline{P}_A)&\approx \psi_a^*\delta(q-Q_a)\delta(p-P_a)\\
    &+ p^*_a\cdot\nabla_q\left[\delta(q-Q_a)\delta(p-P_a)\right] - q^*_a\cdot\nabla_p\left[\delta(q-Q_a)\delta(p-P_a)\right],
\end{align*}
where
\begin{equation}
\label{eq:moment-def}
p^*_a = \sum_{A \in S_a} \overline{\psi}^*_A \; c(Q_a,\overline{Q}_A), 
\qquad
q^*_a = -\sum_{A \in S_a} \overline{\psi}^*_A \; (P_a - \overline{P}_A).
\end{equation}
Here $c(Q_a,\overline{Q}_A) := Q_a - \overline{Q}_A$.

These quantities represent the spatial and momentum dipole moments carried by each decorated particle and provide the initial conditions for the corresponding dual variables $(q_a^*,p_a^*)$ in the Scovel--Weinstein phase space. In particular, $p_a^*$ corresponds to the spatial dipole moment, while $q_a^*$ corresponds to the momentum dipole moment. Each decorated particle is thus initialized by the tuple $(Q_a, P_a, q_a^*, p_a^*, \psi_a^*)$, which serves as the initial state for integrating the finite-dimensional Hamiltonian system introduced in the preceding sections. For large-scale simulations, MiniBatch $k$-means may be used in place of standard $k$-means to obtain nearly identical results at substantially lower computational cost. The complete initialization procedure is summarized in Algorithm~\ref{alg:lumping_procedure}.

\begin{algorithm}[htp]
\caption{Decorated Particle Initialization}
\label{alg:lumping_procedure}
\begin{algorithmic}[1]
\REQUIRE PIC marker positions $\overline{Q}_A$, momenta $\overline{P}_A$,
weights $\overline{\psi}_A^*$, number of clusters $N_p$, domain length $L$
\STATE Group particles in $(q,p)$ into $N_p$ clusters using $k$-means
\STATE Initialize an empty list of $N_p$ decorated particles
\FOR{$a = 1$ to $N_p$}
  \STATE Let $S_a$ be the set of particles assigned to cluster $a$
  \IF{$S_a = \emptyset$}
    \STATE \textbf{continue}
  \ENDIF
  \STATE Find $A_0\in S_a$ for central particle
  \STATE $\displaystyle Q_a = \overline{Q}_{A_0}$; $\displaystyle P_a =\overline{P}_{A_0}$; $\displaystyle \psi_a^* = \sum_{A\in S_a} \overline{\psi}_A^*$; 
  \STATE $\displaystyle p_a^* =
  \sum_{A\in S_a} \overline{\psi}_A^*\,c(Q_a,\overline{Q}_A)$;
  $\displaystyle q_a^* =
  \sum_{A\in S_a}
  -\overline{\psi}_A^*(P_a-\overline{P}_A)$
  \STATE Append $(Q_a,P_a,q_a^*,p_a^*,\psi_a^*)$
\ENDFOR
\RETURN all appended decorated particles
\end{algorithmic}
\end{algorithm}

Note that we do not use the true mean to define the center of the cluster. The reason for this can be understood by considering the simple case of a two-particle cluster. If $(Q_a,P_a) = \overline{Z}_a$ then the above formulas imply $q_a^* = p_a^* = 0$. This implies that the Scovel--Weinstein ansatz collapses to the ordinary PIC ansatz. There is no way to recover more than the mean phase space locations of the particle pair. However, by setting $(Q_a,P_a) = (\overline{Q}_1,\overline{P}_1)$ instead the non-zero values of $q_a^*,p_a^*$ can be used to approximately recover the locations of both particles. The first particle location is of course $(\overline{Q}_1,\overline{P}_1) = (Q_a,P_a)$. Using the observation from Section \ref{sec:singlelump-dynamics} that $(q^*,p^*)\approx (-\delta P,\delta Q)$, where $(\delta Q,\delta P)$ represents a displacement vector, the second particle is located approximately at $(\overline{Q}_2,\overline{P}_2)\approx (Q_a+p^*_a,P_a-q^*_a)$. In other words, displacing the central particle from the true mean allows the ($k=2$) Scovel-Weinstein ansatz to encode more information about the underlying particle distribution than it would with the central particle placed at the true mean. That said, it is also important that the central particle is not displaced too far from the true mean in order to minimize errors in higher moments. The numerical experiments we describe below suggest that our choice of colocating the central particle with the most central PIC particle strikes a good balance between these competing priorities. It will be interesting to analyze optimal placement of the central particle in future work.

\subsection{Initialization with $\mathcal{Q} = \mathbb{T}^d$} The initialization proceeds in nearly the same manner. The crucial difference occurs when attempting to approximate the cluster's empirical distribution with a Scovel--Weinstein ansatz. In the unbounded case Taylor expansion leads to simple expressions for the $(q^*_a,p^*_a)$. But Taylor expanding the empirical distribution in $q$ produces a non-periodic function. Therefore Taylor expanding in $q$ should be avoided when $\mathcal{Q} = \mathbb{T}^d$. To sidestep this issue we use a first-order Fourier series expansion as follows. Let 
\begin{align*}
    f^i(q) = \sin\left(\frac{2\pi}{L}[q^i-Q_a^i]\right),\quad i=1,\dots, d.
\end{align*}
To determine the coefficients $q^*_a$ and $p^*_a$ we impose the following requirements
\begin{align*}
    &\int f^i(q)\,\bigg(\sum_{A\in S_a}\overline{\psi}_A^*\,\delta(q-\overline{Q}_A)\delta(p-\overline{P}_A)\bigg)\,dq\,dp\\
    = & \int f^i(q)\bigg(\psi^*_a\delta(q-{Q}_a)\delta(p-{P}_a)\\
    + & \; p_a^*\cdot\nabla_q[\delta(q-{Q}_a)\delta(p-{P}_a)]
    - q_a^*\cdot\nabla_p[\delta(q-{Q}_a)\delta(p-{P}_a)]\bigg)\,dq\,dp,\quad i=1,\dots, d,
\end{align*}
and
\begin{align*}
    &\int p_i\,\bigg(\sum_{A\in S_a}\overline{\psi}_A^*\,\delta(q-\overline{Q}_A)\delta(p-\overline{P}_A)\bigg)\,dq\,dp\nonumber\\ =& \int p_i\bigg(\psi^*_a\delta(q-{Q}_a)\delta(p-{P}_a)\nonumber\\
    +&\; p_a^*\cdot\nabla_q[\delta(q-{Q}_a)\delta(p-{P}_a)]- q_a^*\cdot\nabla_p[\delta(q-{Q}_a)\delta(p-{P}_a)]\bigg)\,dq\,dp,\quad i=1,\dots, d.
\end{align*}
These requirements result in the expressions \eqref{eq:moment-def} for $q_a^*,p_a^*$, but with $c$ redifined according to
\begin{align*}
    c(Q_a,\overline{Q}_A) := \frac{L}{2\pi}\sin\bigg(\frac{2\pi}{L}[Q^i_a - \overline{Q}_A^i]\bigg)\,\bm{e}_i.
\end{align*}
Note that $c(Q_a,\overline{Q}_A)$ is manifestly periodic. Also note that for a very narrow cluster we have $c(Q_a,\overline{Q}_A)\approx Q_a - \overline{Q}_A$. This indicates matching Fourier moments in a periodic domain is a good substitute for matching ordinary moments in an unbounded domain.
In summary, Algorithm~\ref{alg:lumping_procedure} is applicable for $\mathcal{Q} = \mathbb{T}^d$ with the updated $c(Q_a,\overline{Q}_A)$.

\section{Discretization in space and time}
\label{sec:discrete-hamiltonian} 
A practical implementation of the Scovel--Weinstein model requires specifying the subspace $\mathcal{E}\subset H^1(Q)$ in~\eqref{eq:swpic_poisson} and a time integrator. 
The spatial and temporal discretizations we use in our numerical experiments are outlined below.

For the remainder of the paper we restrict attention to the normalized 1D1V VP system
with $e = m = \epsilon_0 = 1$ on a periodic domain $[0, L)$ (equivalently, on the torus
$\mathbb{T}$), in which the electrostatic potential satisfies the Poisson equation
\begin{equation}
-\varphi''(q) \;=\; \rho(q) - n_0,
\qquad \int_0^L \bigl(\rho(q) - n_0\bigr)\, dq \;=\; 0.
\label{eq:poisson}
\end{equation}
The charge density is given by
\begin{equation}
\rho(q)
\;=\;
\int_{\mathbb{R}} f(q, p, t)\, dp
\;=\;
\sum_{a=1}^{N_p}
\Bigl(
\psi_a^{\ast}\, \delta(q - Q_a)
\;+\;
p_a^{\ast}\delta'(q - Q_a)
\Bigr),
\label{eq:rho}
\end{equation}
with $N_p$ decorated particles located at positions $Q_a \in [0, L)$. The
compatibility condition in \eqref{eq:poisson} then imposes the neutrality
constraint
\begin{equation}
n_0\, L \;=\; \sum_{a=1}^{N_p} \psi_a^{\ast}.
\label{eq:neutrality}
\end{equation}
For simplicity we assume $n_0$ to be constant; both the numerical scheme and its
analysis extend with only cosmetic changes to a $q$-dependent $n_0$.

To approximate \eqref{eq:poisson} we adopt continuous Galerkin (CG) finite elements:
this choice is natural in our PIC setting, where the same Lagrange basis serves both
the standard PIC scheme and the proposed SWPIC scheme, simplifying their direct
comparison. Let $\mathcal{T}_h = \{0 = q_0 < q_1 < \cdots < q_{N-1} < q_N = L\}$ be a
quasi-uniform partition of $[0, L)$ with maximum element size $h$, and let $\mathcal{E} = V_h^{(k)}
\subset H^1(\mathbb{T})$ denote the standard Lagrange finite element space of degree
$k$, augmented with the zero-mean constraint $\int_0^L v_h\,dq = 0$. The
Galerkin approximation $\varphi_h \in V_h^{(k)}$ of \eqref{eq:poisson} solves
\begin{equation}
\int_0^L \varphi_h'(q)\, v_h'(q)\, dq
\;=\; \sum_{a=1}^{N_p}\Bigl[\psi_a^{\ast}\, v_h(Q_a) \;-\; p_a^{\ast}\, v_h'(Q_a)\Bigr],
\qquad \forall\, v_h \in V_h^{(k)},
\label{eq:galerkin}
\end{equation}
where the right-hand side is well defined provided no particle location $Q_a$ is a
mesh node, and the background contribution in \eqref{eq:galerkin} vanishes by the
zero-mean constraint $\int_0^L v_h\,dq = 0$.

Expanding $\varphi_h(q) = \sum_{j=0}^{N-1} \Phi_j\, \phi_j(q)$ in the Lagrange nodal
basis $\{\phi_j\}_{j=0}^{N-1}$ of $V_h^{(k)}$, the Galerkin
system~\eqref{eq:galerkin} becomes the linear system
\begin{equation}
A\,\Phi \;=\; b, \qquad
A_{ij} \;=\; \int_0^L \phi_i'(q)\,\phi_j'(q)\, dq,
\label{eq:lin-sys}
\end{equation}
where $A \in \mathbb{R}^{N \times N}$ is the symmetric positive semidefinite periodic
stiffness matrix with one-dimensional kernel $\mathrm{span}\{(1, \dots, 1)\}$. The right-hand side
vector $b$ assembles the monopole and dipole contributions at the particle locations,
\begin{equation}
b_j \;=\; \sum_{a=1}^{N_p} \Bigl[ \psi_a^{\ast}\, \phi_j(Q_a)
\;-\; p_a^{\ast}\, \phi_j'(Q_a) \Bigr],
\label{eq:rhs-assembly}
\end{equation}
in which $\phi_j(Q_a)$ and $\phi_j'(Q_a)$ are evaluated at the particle position $Q_a$.
The linear system \eqref{eq:lin-sys} is solved subject to the zero-mean gauge
$\mathbf{1}^{\!\top}\!\Phi = 0$.

Substituting the finite element expansion $\varphi_h(q) = \sum_{j=0}^{N-1} \Phi_j\,
\phi_j(q)$ into~\eqref{eq:sw_hamiltonian} yields the spatially-discretized Scovel-Weinstein
Hamiltonian
\begin{equation}
\label{eq:Hh}
H_h
\;=\;
\sum_{a=1}^{N_p}
\left(
\frac{\psi_a^*}{2m}\, P_a^{\,2}
\;+\;
\frac{1}{m}\, q_a^*\, P_a
\right)
\;+\;
\tfrac{1}{2}\,\Phi^{\!\top} A\,\Phi.
\end{equation}
In the 1D1V setting, the noncanonical Poisson bracket on the $k=2$ single-decorated-particle phase-space reads
\begin{align}
\{f,g\}_2
&= \frac{\partial f}{\partial Q}\frac{\partial g}{\partial q^*}
 - \frac{\partial g}{\partial Q}\frac{\partial f}{\partial q^*}
 + \frac{\partial f}{\partial P}\frac{\partial g}{\partial p^*}
 - \frac{\partial g}{\partial P}\frac{\partial f}{\partial p^*}
  \notag\\
&\quad
-\,\psi^*\!\left(
   \frac{\partial f}{\partial q^*}\frac{\partial g}{\partial p^*}
 - \frac{\partial g}{\partial q^*}\frac{\partial f}{\partial p^*}\right).
\label{eq:lp1d}
\end{align}
The Hamilton's equations generated by \eqref{eq:Hh}
under the Poisson bracket \eqref{eq:lp1d} give
\begin{subequations}
\label{eq:sw-evolution}
\begin{align}
\dot{Q}_a &= \frac{P_a}{m},
& \dot{P}_a &= -\varphi_h'(Q_a),
& \dot{\psi}_a^* &= 0, \\
\dot{p}_a^* &= -\,\frac{q_a^*}{m},
& \dot{q}_a^* &= p_a^*\,\varphi_h''(Q_a).
\end{align}
\end{subequations}
In our implementation the first derivative $ \varphi_h'(Q_a)$ is evaluated directly from the Lagrange
basis, while the second derivative $ \varphi_h''(Q_a)$ is reconstructed through an auxiliary projection
step.

Finally, the semi-discrete decorated particle system~\eqref{eq:sw-evolution}, coupled to the
Poisson solve~\eqref{eq:galerkin} at each time step, is advanced in time by a leapfrog
integrator, exactly as in standard PIC
implementations~\cite{birdsall2018plasma}.

\section{Numerical analysis}
\label{sec:numerical-analysis}
We now analyze the convergence of the CG approximation
\eqref{eq:galerkin}--\eqref{eq:rhs-assembly} to the exact solution of
\eqref{eq:poisson} under mesh refinement. 
The singularity in the charge density
\eqref{eq:rho} makes the exact potential only
piecewise smooth. This loss of regularity degrades the standard $L^2$ convergence rate of CG
elements, as analyzed below.

For a given function $g$, let
$
[g]_{Q} \;:=\; g(Q^{+}) - g(Q^{-})
$
denote its jump across $Q$. Since \eqref{eq:galerkin} is
linear in the particle data, it suffices to analyze the single-source case
$N_p = 1$; the multi-particle estimate then follows (see
Remark~\ref{rem:multi-particle}). The exact solution of the single-source problem is given in the following theorem.
\begin{theorem}
\label{thm:periodic-poisson-1d}
Let $L > 0$, $Q \in [0, L)$, and $\psi^{\ast}, p^{\ast} \in \mathbb{R}$. Consider
$L$-periodic functions on $\mathbb{T}$ satisfying
\begin{equation}
-\varphi''(q) \;=\; \psi^{\ast}\, \delta(q - Q) \;+\; p^{\ast}\, \delta'(q - Q) \;-\; n_0,
\qquad q \in \mathbb{T}, \qquad \int_0^L \varphi\, dq \;=\; 0.
\label{eq:pde}
\end{equation}
Then \eqref{eq:pde} is solvable if and only if $n_0 L = {\psi^{\ast}}$,
in which case its unique solution is
\begin{equation}
\varphi(q) \;=\;
\psi^{\ast}\!\left[\frac{r(q)^{2}}{2L} - \frac{|r(q)|}{2} + \frac{L}{12}\right]
\;+\;
p^{\ast}\!\left[\frac{r(q)}{L} - \frac{1}{2}\, \operatorname{sgn} r(q)\right],
\label{eq:solution}
\end{equation}
where $r(q)$ denotes the unique value in $(-L/2,\, L/2]$ satisfying $r(q) \equiv q - Q \pmod{L}$. The solution satisfies the interface jump relations
\begin{equation}
[\varphi]_{Q} \;=\; -p^{\ast}, \qquad [\varphi']_{Q} \;=\; -\psi^{\ast}.
\label{eq:jumps}
\end{equation}
\end{theorem}

The corresponding single-source Galerkin problem is the $N_p = 1$ specialization of
\eqref{eq:galerkin}, namely
\[
\int_0^L \varphi_h'(q)\, v_h'(q)\, dq
\;=\; \psi^{\ast}\, v_h(Q) \;-\; p^{\ast}\, v_h'(Q),
\qquad \forall\, v_h \in V_h^{(k)}.
\]
The main convergence result for the discretization
\eqref{eq:galerkin}--\eqref{eq:rhs-assembly} reads as follows.
\begin{theorem}[$L^2$ convergence]
\label{thm:strongL2}
Given $p^{\ast} \neq 0$, let $\varphi$ be the exact solution \eqref{eq:solution} and 
$\varphi_h \in V_h^{(k)}$ be its CG approximation defined by
\eqref{eq:galerkin} on a quasi-uniform mesh of size $h$ with $Q$ not a mesh node. Then
there exists a constant $C> 0$, independent of $h$, such that
\begin{equation}
\|\varphi - \varphi_h\|_{L^2} \;\leq\;
C\, h^{1/2}\, \bigl(|\psi^{\ast}| + |p^{\ast}|\bigr).
\label{eq:strong-L2-rate}
\end{equation}
The rate $h^{1/2}$ is sharp: it coincides with the best $L^2$ approximation of a step
function by CG bases and cannot be improved by raising the polynomial degree $k$.
\end{theorem}

\begin{proof}
By linearity of \eqref{eq:pde}, decompose
$\varphi = \varphi_M + \varphi_D$ into the the monopole and dipole contributions
\[
\varphi_M(q) \;=\; \psi^{\ast}\!\left[\tfrac{r(q)^2}{2L} - \tfrac{|r(q)|}{2} + \tfrac{L}{12}\right],
\qquad
\varphi_D(q) \;=\; p^{\ast}\!\left[\tfrac{r(q)}{L} - \tfrac{1}{2}\operatorname{sgn} r(q)\right],
\]
with $\varphi_h^M, \varphi_h^D \in V_h^{(k)}$ defined as the Galerkin solutions
of \eqref{eq:galerkin} with right-hand sides $\psi^{\ast} v_h(Q)$ and $-p^{\ast} v_h'(Q)$,
respectively. The two contributions are estimated separately.

For $\varphi_M$, since we have $\varphi_M \in H^1(\mathbb{T})$ and the Galerkin problem
for $\varphi_M$ falls within the classical situation, 
so C\'{e}a's lemma yields
\[
\|\varphi_M - \varphi_h^M\|_{H^1} \;\leq\; C\, h^{1/2}\, |\psi^{\ast}|.
\]
The Aubin--Nitsche duality argument is applicable here, gaining one power of $h$:
\begin{equation}
\|\varphi_M - \varphi_h^M\|_{L^2}
\;\leq\; C\, h^{3/2}\, |\psi^{\ast}|.
\label{eq:proof-monopole}
\end{equation}

For $\varphi_D$, we exploit the 1D structure to construct $\varphi_h^D$ explicitly; we
present the argument for $k = 1$, the case $k \geq 2$ following by an analogous 
computation. Let $e_Q := [q_{i_Q}, q_{i_Q+1}]$ denote the element of length $h_{e_Q}$
containing $Q$. The discrete solution $\varphi_h^D \in V_h^{(1)}$ is the continuous
piecewise linear function with slope $p^{\ast}/L$ on every element $e \neq e_Q$ and slope
$p^{\ast}/L - p^{\ast}/h_{e_Q}$ on $e_Q$, normalized by the zero-mean constraint. In
particular, the off-$e_Q$ slope coincides with the continuous slope of $\varphi_D$ on
$\mathbb{T} \setminus \{Q\}$. A direct elementwise check confirms that this
$\varphi_h^D$ satisfies
$\int_0^L (\varphi_h^D)'\, v_h'\, dq = -p^{\ast}\, v_h'(Q)$
for every $v_h \in V_h^{(1)}$.

Setting $\eta := \varphi_D - \varphi_h^D$, both $\varphi_D$ and $\varphi_h^D$ are linear with
slope $p^{\ast}/L$ off $e_Q$, so $\eta$ is constant on $\mathbb{T} \setminus e_Q$; call this
constant $C_0$. With $a := Q - q_{i_Q}$, $b := q_{i_Q+1} - Q$, direct
evaluation gives
\begin{equation}
C_0 = \frac{p^{\ast}(b - a)}{2L}, \qquad |C_0| \leq \frac{|p^{\ast}|\, h_{e_Q}}{2L}.
\nonumber
\end{equation}
On $e_Q$, $\varphi_h^D$ is linear connecting $\varphi_h^D(q_{i_Q}) = p^{\ast}(L - h_{e_Q})
/(2L)$ to $\varphi_h^D(q_{i_Q+1}) = -p^{\ast}(L - h_{e_Q})/(2L)$, while $\varphi_D$ is
piecewise linear with the value jump $[\varphi_D]_Q = -p^{\ast}$. Hence $\|\eta\|_{L^{\infty}
(e_Q)} \leq C |p^{\ast}|$ and
\begin{equation}
\|\eta\|_{L^2(e_Q)} \;\leq\; C\,|p^{\ast}|\, h_{e_Q}^{1/2}.
\nonumber
\end{equation}
Combining,
\begin{equation}
\|\varphi_D - \varphi_h^D\|_{L^2(\mathbb{T})}^2
\;=\; (L - h_{e_Q})\,C_0^2 + \|\eta\|_{L^2(e_Q)}^2
\;\leq\; C_1 h_{e_Q}^2 + C_2 |p^{\ast}|^2\, h_{e_Q}
\;\leq\; C\,|p^{\ast}|^2\, h,
\label{eq:proof-dipole}
\end{equation}
yielding
\begin{equation}
\|\varphi_D - \varphi_h^D\|_{L^2(\mathbb{T})} \;\leq\; C\, h^{1/2}\, |p^{\ast}|.
\nonumber
\end{equation}

Note that the rate is sharp: any $v_h \in V_h^{(k)}$ satisfies
$[v_h]_Q = 0$, so $|\varphi_D - v_h| \geq |p^{\ast}|/2$ on a subset of $e_Q$, 
giving $\|\varphi_D - v_h\|_{L^2(\mathbb{T})} \geq c\, h^{1/2}\,
|p^{\ast}|$ for all $v_h \in V_h^{(k)}$.

Combining \eqref{eq:proof-monopole} and \eqref{eq:proof-dipole} via the triangle
inequality,
\[
\|\varphi - \varphi_h\|_{L^2}
\;\leq\; \|\varphi_M - \varphi_h^M\|_{L^2} + \|\varphi_D - \varphi_h^D\|_{L^2}
\;\leq\; C\, h^{1/2}\, \bigl(|\psi^{\ast}| + |p^{\ast}|\bigr).
\]
The sharpness statement above shows that the rate $h^{1/2}$ cannot be improved by raising
$k$.
\end{proof}

\begin{remark}[Extension to multiple particles]
\label{rem:multi-particle}
Theorem~\ref{thm:strongL2} extends to the multi-particle source \eqref{eq:rho} with $N_p$
decorated particles, provided the particle
locations $Q_a$ are distinct and none coincides with a mesh node. By linearity of
\eqref{eq:pde} and \eqref{eq:galerkin} and the per-particle bound combined  gives
\begin{equation}
\|\varphi - \varphi_h\|_{L^2}
\;\leq\; C\, h^{1/2} \sum_{a=1}^{N_p} \bigl(|\psi_a^{\ast}| + |p_a^{\ast}|\bigr).
\nonumber
\end{equation}
The constant remains uniformly bounded in $N_p$ due to physics, where
the per-particle weights scale as $\psi_a^{\ast}, p_a^{\ast} = \mathcal{O}(1/N_p)$, yielding $O(h^{1/2})$
convergence under multiple particles.
\end{remark}

\section{Numerical results}
\label{sec:numerical}
This section presents numerical experiments carried out with the proposed SWPIC framework. We first consider a test particle case to provide a basic illustration of the algorithm. The method is then applied to standard 1D1V VP tests, including the two-stream instability and Landau damping. All experiments use a periodic boundary condition.

\subsection{Test particle}
\label{sec:toy-example}
To illustrate the Hamiltonian dynamics captured by SWPIC, we consider a simple test particle problem. We take \(N=30\) particles moving under the periodic potential
\begin{equation*}
V(q)=1-\cos(\kappa q),
\qquad
\kappa=\frac{2\pi}{L},
\end{equation*}
on a domain of length \(L=10\).  The particles are initialized in three groups of ten, each centered at a different phase-space location. Although simple, this configuration already exhibits nontrivial collective behavior. This is a multi-particle generalization of the example from Section \ref{sec:singlelump-dynamics}.

In the standard PIC description, the dynamics is determined by the particle positions, momenta, and weights, giving \(3N=90\) degrees of freedom. In SWPIC, these particles are represented by three decorated particles. Each decorated particle is described by central variables \((Q_a,P_a)\) together with moment variables \((q_a^*,p_a^*,\psi_a^*)\), resulting in only \(3\times5=15\) degrees of freedom overall.
The Hamiltonian for this reduced system is
\begin{equation*}
H_{\text{SWPIC}}
=
\sum_{a=1}^{3}
\left(
\frac{\psi_a^*}{2m}P_a^2
+
\frac{1}{m}P_a q_a^*
+
\psi_a^*\,[1-\cos(\kappa Q_a)]
-
p_a^*\,\kappa\,\sin(\kappa Q_a)
\right).
\end{equation*}
The corresponding equations of motion derived from the Scovel--Weinstein Poisson bracket are
\begin{subequations}
\label{eq:swpic_test_particle}
\begin{align}
\dot{Q}_a &= \frac{P_a}{m},
&
\dot{P}_a &= -\,\kappa\,\sin(\kappa Q_a),
&
\dot{\psi}_a^* &= 0,
\label{eq:swpic_test_particle_QP}
\\[4pt]
\dot{p}_a^*
&=
-\,\frac{q_a^*}{m},
&
\dot{q}_a^*
&=
p_a^*\,\kappa^2\,\cos(\kappa Q_a).
\label{eq:swpic_test_particle_qp}
\end{align}
\end{subequations}

\autoref{fig:combined_init_final} compares the PIC and SWPIC trajectories in the prescribed potential. Panels~(a) and~(b) show the initial particle configuration together with its reduction into three decorated particles. At \(t=10\), panels~(c) and~(d) show that the reduced trajectories remain consistent with the collective motion of the corresponding particle groups
whiles \autoref{fig:variables_evolution} shows the evolution of the central variables together with the associated internal variables $(Q_a,P_a,q_a^*,p_a^*,\psi_a^*).$
This example illustrates how SWPIC captures the essential particle dynamics using a significantly smaller set of variables while retaining the underlying geometric structure. 

\begin{figure}[htp]
    \centering
    \begin{tikzpicture}
    \begin{groupplot}[
      group style={
        group size=4 by 1,
        horizontal sep=18pt,
      },
      width=4.0cm,
      height=4.0cm,
      xmin=0, xmax=10, ymin=-4, ymax=4,
      grid=none,
      tick style={black, thick},
      axis line style={thick},
      xlabel style={font=\bfseries\scriptsize, yshift=1pt},
      ylabel style={font=\bfseries\scriptsize, yshift=-1pt},
      every tick label/.append style={font=\scriptsize},
      clip=false
    ]

    \nextgroupplot[xlabel={$q$}, ylabel={$p$}]
      \addplot[only marks, mark=*, mark size=1.6pt, red]
        table[col sep=comma, x=q, y=p, restrict expr to domain={\thisrow{group}}{0:0}]
              {Data/PIC_phase_t0.csv};
      \addplot[only marks, mark=*, mark size=1.6pt, blue]
        table[col sep=comma, x=q, y=p, restrict expr to domain={\thisrow{group}}{1:1}]
              {Data/PIC_phase_t0.csv};
      \addplot[only marks, mark=*, mark size=1.6pt, green!70!black]
        table[col sep=comma, x=q, y=p, restrict expr to domain={\thisrow{group}}{2:2}]
              {Data/PIC_phase_t0.csv};
      \node[anchor=north west, font=\footnotesize, inner sep=2pt] 
        at (rel axis cs:0.02,0.98) {(a)};
      \addplot[forget plot, mark=none] coordinates{(0,0)} node[
          anchor=north east, fill=white, draw=black, rounded corners=1pt,
          line width=0.4pt, inner sep=2pt, font=\footnotesize
      ] at (rel axis cs:0.94,0.94) {\shortstack[c]{PIC\\$t=0$}};

    \nextgroupplot[xlabel={$Q$}, ylabel={}, yticklabels={}]
      \addplot[only marks, mark=*, mark size=2pt, red]
        table[col sep=comma, x=Q, y=P, restrict expr to domain={\thisrow{group}}{0:0}]
              {Data/SWPIC_phase_t0.csv};
      \addplot[only marks, mark=*, mark size=2pt, blue]
        table[col sep=comma, x=Q, y=P, restrict expr to domain={\thisrow{group}}{1:1}]
              {Data/SWPIC_phase_t0.csv};
      \addplot[only marks, mark=*, mark size=2pt, green!70!black]
        table[col sep=comma, x=Q, y=P, restrict expr to domain={\thisrow{group}}{2:2}]
              {Data/SWPIC_phase_t0.csv};
      \node[anchor=north west, font=\footnotesize, inner sep=2pt] 
        at (rel axis cs:0.02,0.98) {(b)};
      \addplot[forget plot, mark=none] coordinates{(0,0)} node[
          anchor=north east, fill=white, draw=black, rounded corners=1pt,
          line width=0.4pt, inner sep=2pt, font=\footnotesize
      ] at (rel axis cs:0.94,0.94) {\shortstack[c]{SWPIC\\$t=0$}};

    \nextgroupplot[xlabel={$q$}, ylabel={}, yticklabels={}]
      \addplot[only marks, mark=*, mark size=1.6pt, red]
        table[col sep=comma, x=q, y=p, restrict expr to domain={\thisrow{group}}{0:0}]
              {Data/PIC_phase_t1.csv};
      \addplot[only marks, mark=*, mark size=1.6pt, blue]
        table[col sep=comma, x=q, y=p, restrict expr to domain={\thisrow{group}}{1:1}]
              {Data/PIC_phase_t1.csv};
      \addplot[only marks, mark=*, mark size=1.6pt, green!70!black]
        table[col sep=comma, x=q, y=p, restrict expr to domain={\thisrow{group}}{2:2}]
              {Data/PIC_phase_t1.csv};
      \node[anchor=north west, font=\footnotesize, inner sep=2pt] 
        at (rel axis cs:0.02,0.98) {(c)};
      \addplot[forget plot, mark=none] coordinates{(0,0)} node[
          anchor=north east, fill=white, draw=black, rounded corners=1pt,
          line width=0.4pt, inner sep=2pt, font=\footnotesize
      ] at (rel axis cs:0.94,0.94) {\shortstack[c]{PIC\\$t=1$}};

    \nextgroupplot[xlabel={$Q$}, ylabel={}, yticklabels={}]
      \addplot[only marks, mark=*, mark size=2pt, red]
        table[col sep=comma, x=Q, y=P, restrict expr to domain={\thisrow{group}}{0:0}]
              {Data/SWPIC_phase_t1.csv};
      \addplot[only marks, mark=*, mark size=2pt, blue]
        table[col sep=comma, x=Q, y=P, restrict expr to domain={\thisrow{group}}{1:1}]
              {Data/SWPIC_phase_t1.csv};
      \addplot[only marks, mark=*, mark size=2pt, green!70!black]
        table[col sep=comma, x=Q, y=P, restrict expr to domain={\thisrow{group}}{2:2}]
              {Data/SWPIC_phase_t1.csv};
      \node[anchor=north west, font=\footnotesize, inner sep=2pt] 
        at (rel axis cs:0.02,0.98) {(d)};
      \addplot[forget plot, mark=none] coordinates{(0,0)} node[
          anchor=north east, fill=white, draw=black, rounded corners=1pt,
          line width=0.4pt, inner sep=2pt, font=\footnotesize
      ] at (rel axis cs:0.94,0.94) {\shortstack[c]{SWPIC\\$t=1$}};

    \end{groupplot}
    \end{tikzpicture}
\caption{Comparison of PIC and SWPIC particle evolution. Panels~(a,b) show the initial configuration: (a) 30 microscopic particles grouped by color and (b) their SWPIC reduction into three decoarted particles obtained via $k$-means. At $t=10$, panels~(c,d) show (c) the full PIC system and (d) decorated particles, which follow the same collective trajectories while preserving Hamiltonian structure with fewer DOFs.}
    \label{fig:combined_init_final}
\end{figure}
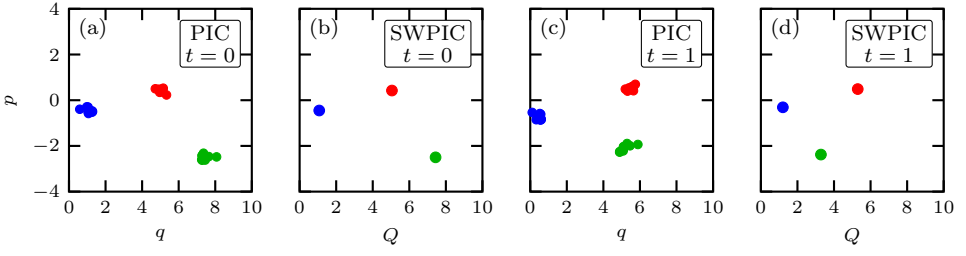

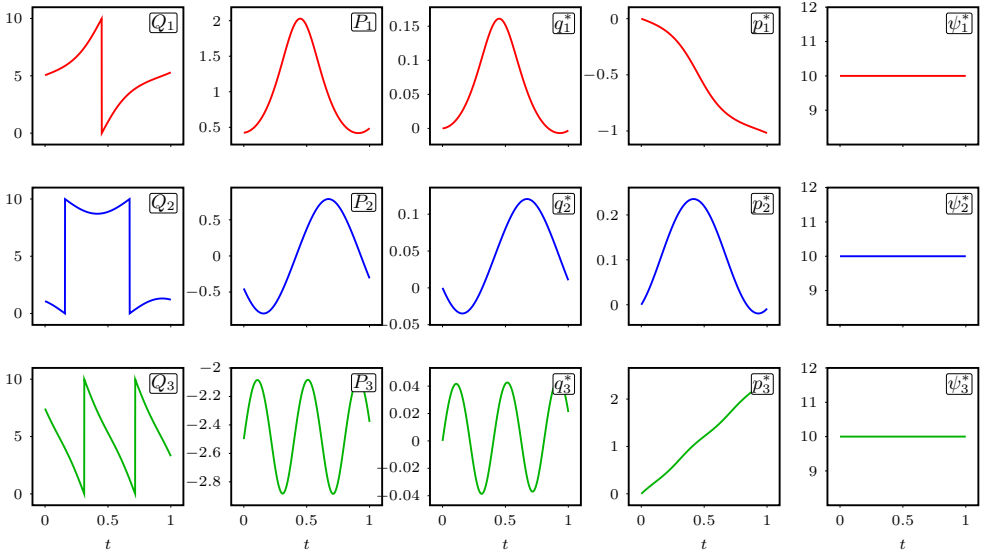
\begin{figure}[htp]
\centering

\tikzset{
  labelbox/.style={
    anchor=north east,
    fill=white,
    draw=black,
    rounded corners=0.8pt,
    line width=0.3pt,
    inner sep=0.6pt,
    font=\small
  }
}

\pgfplotsset{
  lumpaxis/.style={
    width=3.8cm,
    height=3.6cm,
    grid=none,
    tick style={black, thin},
    tick align=outside,
    tick pos=left,
    major tick length=1pt,
    every tick label/.append style={font=\scriptsize},
    axis line style={thick},
    xlabel style={font=\bfseries\scriptsize},
    ylabel={},
    scaled y ticks=false,
    yticklabel style={/pgf/number format/fixed},
    clip=false
  }
}

\begin{tikzpicture}[font=\scriptsize\sffamily, scale=0.9, every node/.style={scale=0.9}]

\begin{groupplot}[
  lumpaxis,
  group style={
    group size=5 by 3,
    horizontal sep=20pt,
    vertical sep=18pt
  }
]


\nextgroupplot[xtick={0,0.5,1}, xticklabels={}, xlabel={}]
\addplot[thick, red] table[col sep=comma, x=time, y=Q]{Data/SWPIC_lump_1.csv};
\node[labelbox] at (rel axis cs:0.96,0.96) {$Q_1$};

\nextgroupplot[xtick={0,0.5,1}, xticklabels={}, xlabel={}]
\addplot[thick, red] table[col sep=comma, x=time, y=P]{Data/SWPIC_lump_1.csv};
\node[labelbox] at (rel axis cs:0.96,0.96) {$P_1$};

\nextgroupplot[xtick={0,0.5,1}, xticklabels={}, xlabel={}]
\addplot[thick, red] table[col sep=comma, x=time, y=q_star]{Data/SWPIC_lump_1.csv};
\node[labelbox] at (rel axis cs:0.96,0.96) {$q_1^{*}$};

\nextgroupplot[xtick={0,0.5,1}, xticklabels={}, xlabel={}]
\addplot[thick, red] table[col sep=comma, x=time, y=p_star]{Data/SWPIC_lump_1.csv};
\node[labelbox] at (rel axis cs:0.96,0.96) {$p_1^{*}$};

\nextgroupplot[xtick={0,0.5,1}, xticklabels={}, xlabel={}]
\addplot[thick, red] table[col sep=comma, x=time, y=psi_star]{Data/SWPIC_lump_1.csv};
\node[labelbox] at (rel axis cs:0.96,0.96) {$\psi_1^{*}$};


\nextgroupplot[xtick={0,0.5,1}, xticklabels={}, xlabel={}]
\addplot[thick, blue] table[col sep=comma, x=time, y=Q]{Data/SWPIC_lump_2.csv};
\node[labelbox] at (rel axis cs:0.96,0.96) {$Q_2$};

\nextgroupplot[xtick={0,0.5,1}, xticklabels={}, xlabel={}]
\addplot[thick, blue] table[col sep=comma, x=time, y=P]{Data/SWPIC_lump_2.csv};
\node[labelbox] at (rel axis cs:0.96,0.96) {$P_2$};

\nextgroupplot[xtick={0,0.5,1}, xticklabels={}, xlabel={}]
\addplot[thick, blue] table[col sep=comma, x=time, y=q_star]{Data/SWPIC_lump_2.csv};
\node[labelbox] at (rel axis cs:0.96,0.96) {$q_2^{*}$};

\nextgroupplot[xtick={0,0.5,1}, xticklabels={}, xlabel={}]
\addplot[thick, blue] table[col sep=comma, x=time, y=p_star]{Data/SWPIC_lump_2.csv};
\node[labelbox] at (rel axis cs:0.96,0.96) {$p_2^{*}$};

\nextgroupplot[xtick={0,0.5,1}, xticklabels={}, xlabel={}]
\addplot[thick, blue] table[col sep=comma, x=time, y=psi_star]{Data/SWPIC_lump_2.csv};
\node[labelbox] at (rel axis cs:0.96,0.96) {$\psi_2^{*}$};


\nextgroupplot[xlabel={$t$}, xtick={0,0.5,1}]
\addplot[thick, green!70!black] table[col sep=comma, x=time, y=Q]{Data/SWPIC_lump_3.csv};
\node[labelbox] at (rel axis cs:0.96,0.96) {$Q_3$};

\nextgroupplot[xlabel={$t$}, xtick={0,0.5,1}]
\addplot[thick, green!70!black] table[col sep=comma, x=time, y=P]{Data/SWPIC_lump_3.csv};
\node[labelbox] at (rel axis cs:0.96,0.96) {$P_3$};

\nextgroupplot[xlabel={$t$}, xtick={0,0.5,1}]
\addplot[thick, green!70!black] table[col sep=comma, x=time, y=q_star]{Data/SWPIC_lump_3.csv};
\node[labelbox] at (rel axis cs:0.96,0.96) {$q_3^{*}$};

\nextgroupplot[xlabel={$t$}, xtick={0,0.5,1}]
\addplot[thick, green!70!black] table[col sep=comma, x=time, y=p_star]{Data/SWPIC_lump_3.csv};
\node[labelbox] at (rel axis cs:0.96,0.96) {$p_3^{*}$};

\nextgroupplot[xlabel={$t$}, xtick={0,0.5,1}]
\addplot[thick, green!70!black] table[col sep=comma, x=time, y=psi_star]{Data/SWPIC_lump_3.csv};
\node[labelbox] at (rel axis cs:0.96,0.96) {$\psi_3^{*}$};

\end{groupplot}
\end{tikzpicture}

\caption{
Time evolution of the Scovel--Weinstein cluster variables for three representative clusters.
Each row shows the trajectories of
\((Q_a,P_a,q_a^*,p_a^*,\psi_a^*)\)
for one cluster, with \(a=1,2,3\) displayed in red, blue, and green.
}
\label{fig:variables_evolution}
\end{figure}

\subsection{Two-stream instability}
\label{sec:twostream}
We apply SWPIC to the warm two-stream instability in 1D1V. The reference PIC simulation uses $10^{5}$ particles, while the SWPIC run employs $10^{4}$ phase-pace clusters obtained from the same microscopic particle realization. All quantities are nondimensionalized with $m_e=\epsilon_0=1$ on a periodic domain of length $L=2\pi$. The initial distribution consists of two symmetric warm beams,
\[
f(t=0,q,p)
= \tfrac12\!\left[f_M(p-1)+f_M(p+1)\right], \quad
f_M(p)=\frac{1}{\sqrt{2\pi T}}\,e^{-p^2/(2T)}, \quad T=0.3^2,
\]
sampled uniformly in $q\in[0,2\pi)$. The SWPIC initial condition is obtained by clustering the sampled particles in $(q,p)$ and assigning each cluster its central variables and first--order internal moments, as described in
Algorithm~\ref{alg:lumping_procedure}. The electric field is computed using a CG Poisson solve with $50$ elements.
Time integration is performed using a leapfrog scheme with $\Delta t=0.01$ for $1000$ steps. \autoref{fig:pic_swpic} shows phase-space snapshots at $t = 0$, $5$, and $10$. Both PIC and SWPIC capture the characteristic distortion of the beams as
the instability develops. The electric–field amplitude,
\begin{equation}
    E_{\mathrm{amp}}(t)
    = \Bigl(\tfrac{1}{L}\!\int_0^L E(q,t)^2\,dq\Bigr)^{1/2},
\label{eq:Eamp}
\end{equation}
is reported in \autoref{fig:twostream_compare}, where
 the growth rate is found to be $\gamma = 0.7863$. Both
methods follow this trend closely. In this setting, SWPIC reproduces the
observed field evolution and phase-space dynamics of the PIC simulation while
using an order of magnitude fewer particles.
\begin{figure}[htp!]
    \centering
    \includegraphics[width=0.33\linewidth]{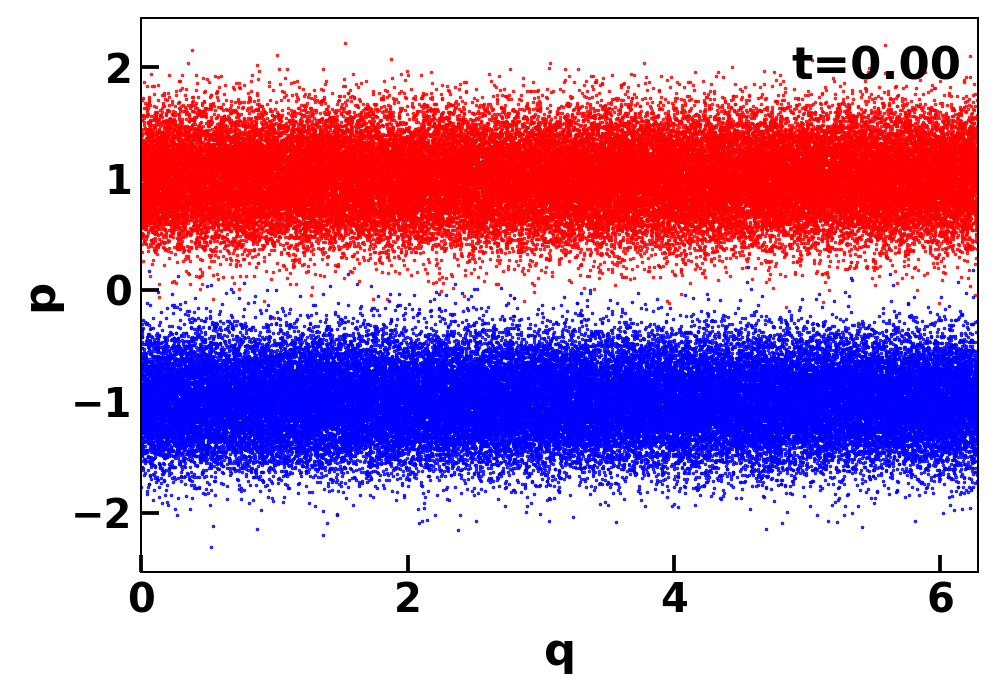}\hfill
    \includegraphics[width=0.33\linewidth]{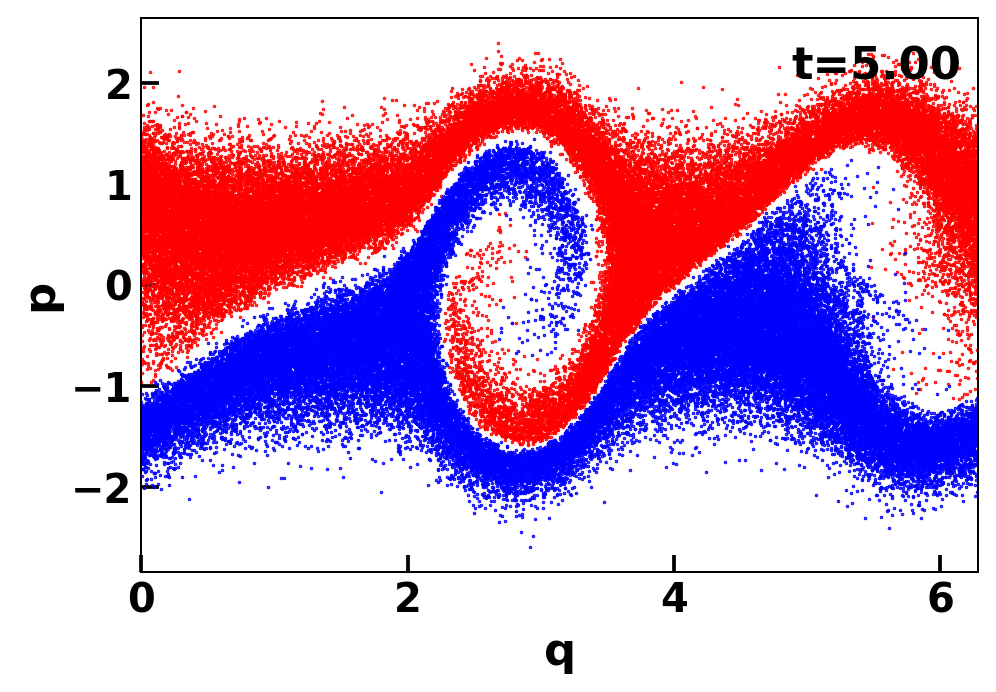}\hfill
    \includegraphics[width=0.33\linewidth]{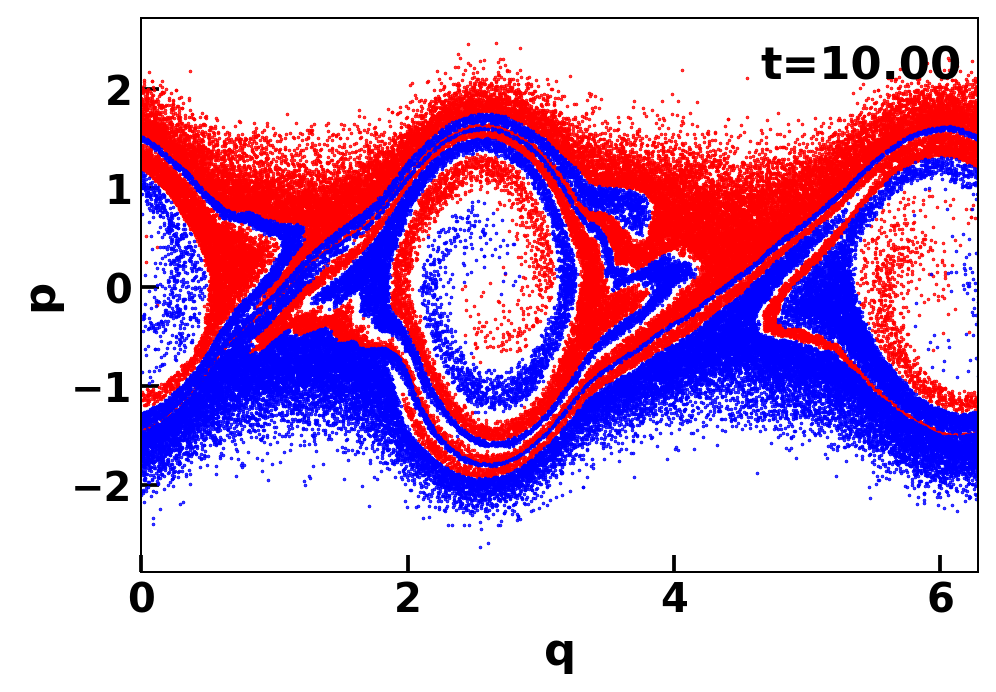} \\

    \includegraphics[width=0.33\linewidth]{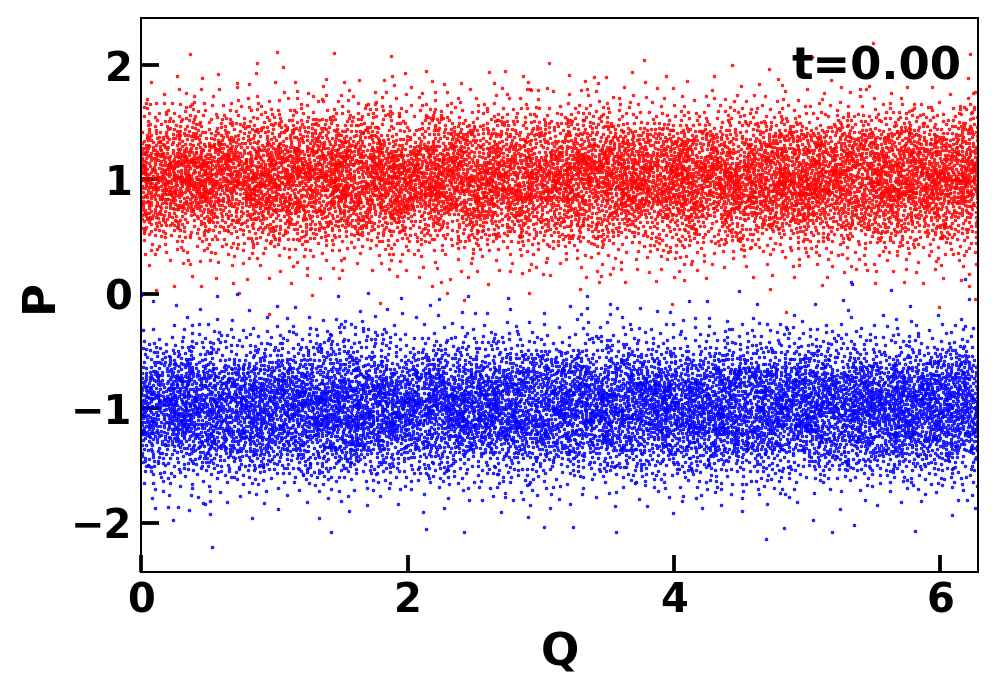}\hfill
    \includegraphics[width=0.33\linewidth]{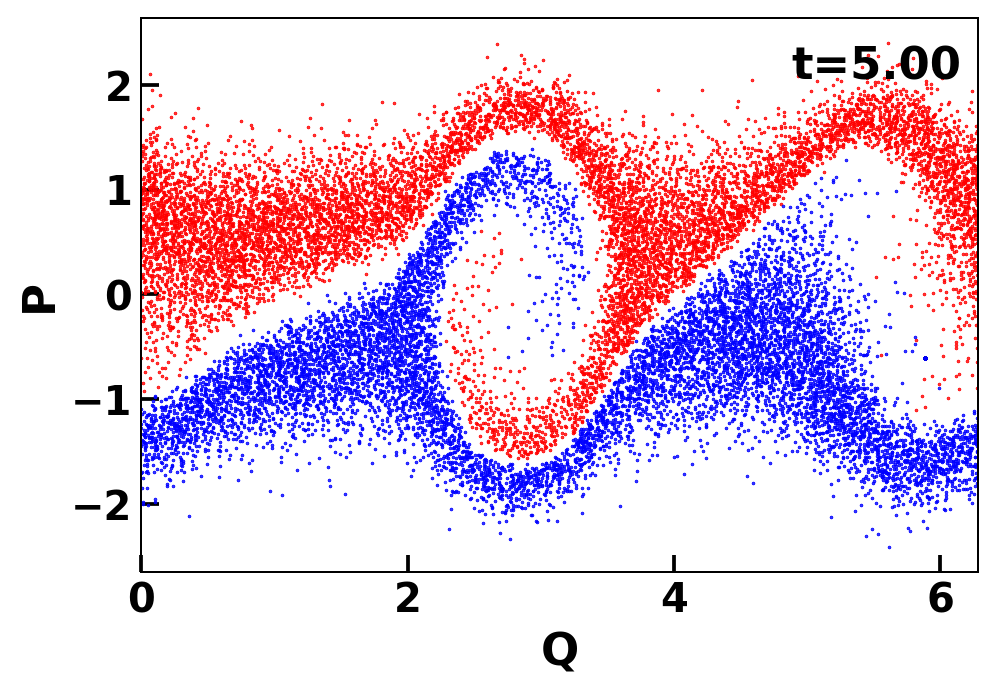}\hfill
    \includegraphics[width=0.33\linewidth]{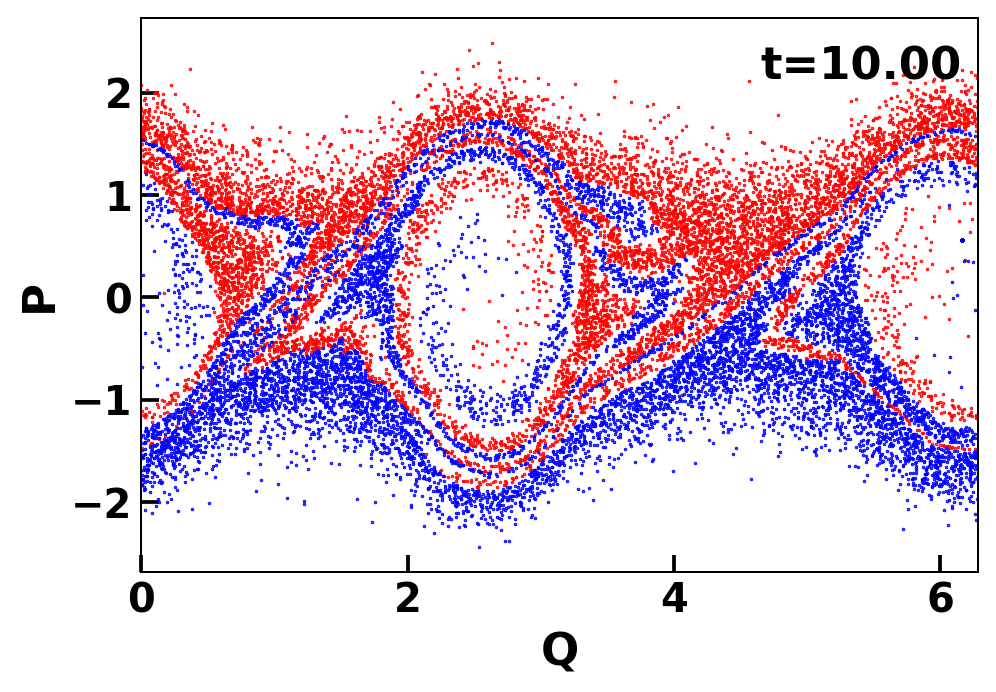} \\
    \caption{Two-stream instability of PIC and SWPIC at $t=0$, $t=5$, and $t=10$. 
The top row shows PIC results with $10^6$ particles, 
and the bottom row shows SWPIC results with $10^5$ decorated particles. 
SWPIC closely reproduces the nonlinear phase-space structures and mixing observed in PIC, capturing the full instability dynamics with fewer DOFs.}
    \label{fig:pic_swpic}
\end{figure}

\begin{figure}[htp!]
    \centering
    \begin{tikzpicture}
        \begin{axis}[
            width=0.6\linewidth,
            height=0.43\linewidth,
            xlabel={Time},
            ylabel={$E$-field amplitude},
            xlabel style={font=\bfseries\small, yshift=2pt},
            ylabel style={font=\bfseries\small, yshift=-2pt},
            ymode=log,
            grid=both,
            tick style={black, thin},
            line width=0.9pt,
            legend style={
                at={(0.97,0.03)},
                anchor=south east,
                draw=black,
                line width=0.4pt,
                rounded corners=2pt,
                font=\footnotesize,
                row sep=1pt,
                /tikz/every even column/.style={column sep=3pt},
                legend image post style={xscale=0.9, yscale=0.9}
            },
            legend cell align=left,
            enlarge x limits=false,
            enlarge y limits=false
        ]
            \addplot[
                thick,
                red,
                solid,
                opacity=0.9
            ]
            table[x=t_pic, y=amp_pic, col sep=comma]
            {Data/twostream_compare.csv};
            \addlegendentry{PIC}
            
            \addplot[
                thick,
                blue,
                dashed,
                opacity=0.9
            ]
            table[x=t_swpic, y=amp_swpic, col sep=comma]
            {Data/twostream_compare.csv};
            \addlegendentry{SWPIC}
            
            \addplot[
                thick,
                black,
                dash pattern=on 3pt off 2pt,
                opacity=0.9
            ]
            table[x=t_fit, y=fit_y, col sep=comma]
            {Data/twostream_compare.csv};
            \addlegendentry{$\gamma=0.7863$}
        \end{axis}
    \end{tikzpicture}
    \caption{$E$-field amplitude for the two-stream instability in PIC and SWPIC. 
Both simulations exhibit identical exponential growth during the linear phase and reach the same saturation level. 
The fitted line corresponds to a growth rate of $\gamma=0.7863$.}
    \label{fig:twostream_compare}
\end{figure}
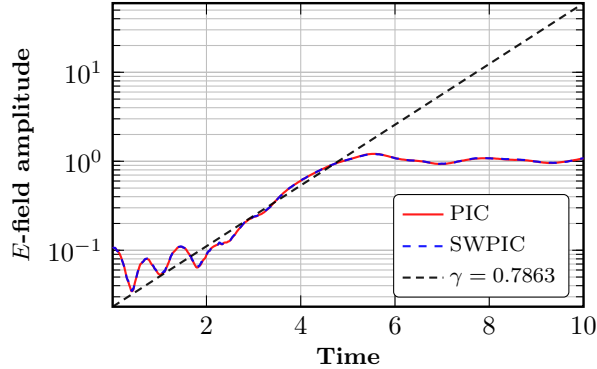

\subsection{Strong Landau damping}
\label{sec:landau} 
We next examine the performance of SWPIC in the strong Landau damping regime. The reference PIC simulation uses $10^{5}$ particles, while the SWPIC simulation employs
$1\times10^{4}$ phase-space clusters obtained from the same microscopic realization. All quantities are nondimensionalized with $m_e=\epsilon_0=1$ on a periodic domain of length $L=12$. The initial condition is a perturbed Maxwellian,
\[
f(t=0,q,p)=f_M(p)\,[1+A\cos(kq)], \qquad
f_M(p)=\tfrac{1}{\sqrt{2\pi}}\,e^{-p^2/2},
\]
with $A=0.5$ and $k=2\pi/L$.
Particles are sampled uniformly in $q$ and from the corresponding velocity distribution. The SWPIC initialization follows Algorithm~\ref{alg:lumping_procedure}. The electric field is obtained from a CG Poisson solve with $100$ grid points. The step of $\Delta t=0.2$ is used for evolving $500$ steps, and the field amplitude~\eqref{eq:Eamp} is monitored.

\autoref{fig:landau_phase_compare} shows the evolution of $E_{\mathrm{amp}}(t)$, giving a damping rate of $\gamma=-0.236$, in agreement with established nonlinear values~\cite{qiu2010conservative}. By $t=25$, both PIC and SWPIC produce comparable phase–space structures following the initial decay. The temporal behavior of the field amplitude agrees closely between the two methods. A slight phase offset becomes visible only after the field has decayed to its small-amplitude remnant, a regime where the dynamics are highly sensitive to sampling variability and to the reduced number of decoarted particles. This late-time drift is expected in strong Landau damping, where the residual field is extremely weak and small differences accumulate over many oscillations, and it does not seem to affect the extracted damping rate or the overall agreement between the two methods. 
In this setting, SWPIC reproduces the qualitative dynamics observed
in the full PIC system while employing an order of magnitude fewer particles.

\begin{figure}[htp]
	\centering
	\noindent
	\begin{minipage}[b]{0.62\linewidth}
		\centering
\resizebox{\linewidth}{!}{%
  \begin{tikzpicture}
        \begin{axis}[
            width=\linewidth,
            height=.9\linewidth,
            xlabel={Time},
            ylabel={$E$-field amplitude},
            xlabel style={font=\bfseries\small, yshift=2pt},
            ylabel style={font=\bfseries\small, yshift=-2pt},
            ymode=log,
            grid=both,
            tick style={black, thin},
            line width=0.9pt,
            legend style={
                at={(0.97,0.03)},
                anchor=south east,
                draw=black,
                line width=0.4pt,
                rounded corners=2pt,
                font=\footnotesize,
                row sep=1pt,
                /tikz/every even column/.style={column sep=3pt},
                legend image post style={xscale=0.9, yscale=0.9}
            },
            legend cell align=left,
            xmin=0, xmax=100,
        ]
            \addplot[
                thick,
                red,
                solid,
                opacity=0.9
            ]
            table[x=t_pic, y=amp_pic, col sep=comma]
            {Data/landau_damping_compare.csv};
            \addlegendentry{PIC}
            
            \addplot[
                thick,
                blue,
                dashed,
                opacity=0.9
            ]
            table[x=t_swpic, y=amp_swpic, col sep=comma]
            {Data/landau_damping_compare.csv};
            \addlegendentry{SWPIC}
            
            \addplot[
                thick,
                black,
                dash pattern=on 3pt off 2pt,
                opacity=0.9
            ]
            table[x=t_fit, y=fit_y, col sep=comma]
            {Data/landau_damping_compare.csv};
            \addlegendentry{$\gamma=-0.236$}
        \end{axis}
    \end{tikzpicture}
}\\
\vspace{.1in}
    {(a) $E$-field amplitude in time}
	\end{minipage}%
	\hspace{0.03\linewidth}
	\begin{minipage}[b]{0.34\linewidth}
		\centering
		\includegraphics[width=\linewidth]{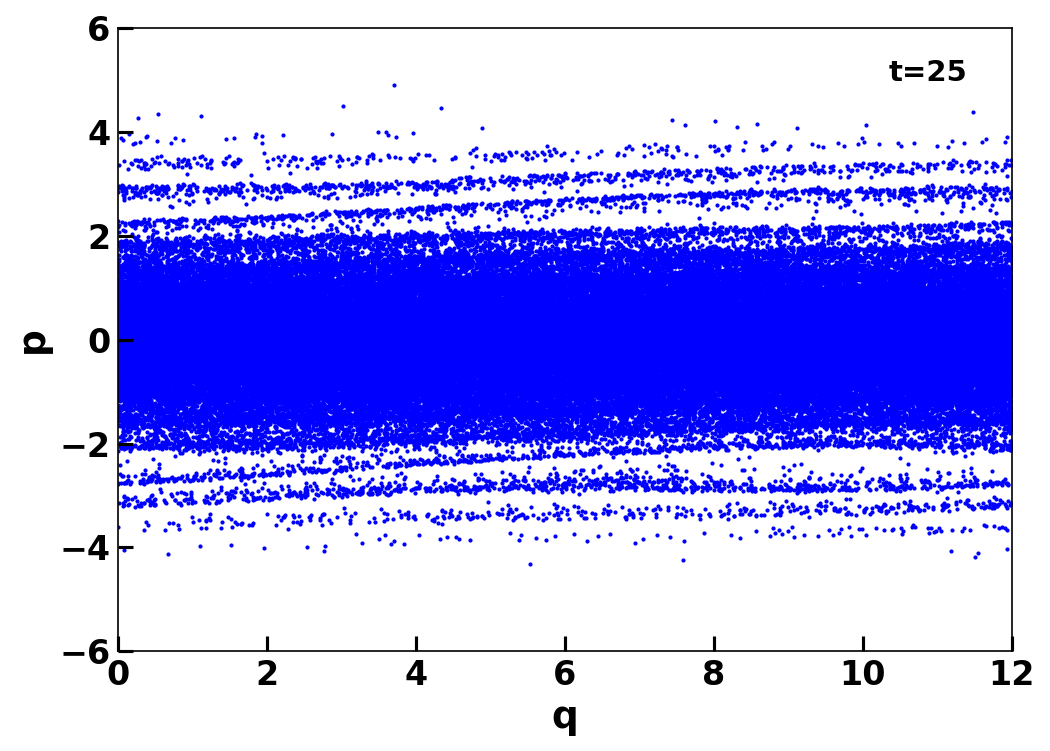}\\
		{(b) PIC}\\
		\includegraphics[width=\linewidth]{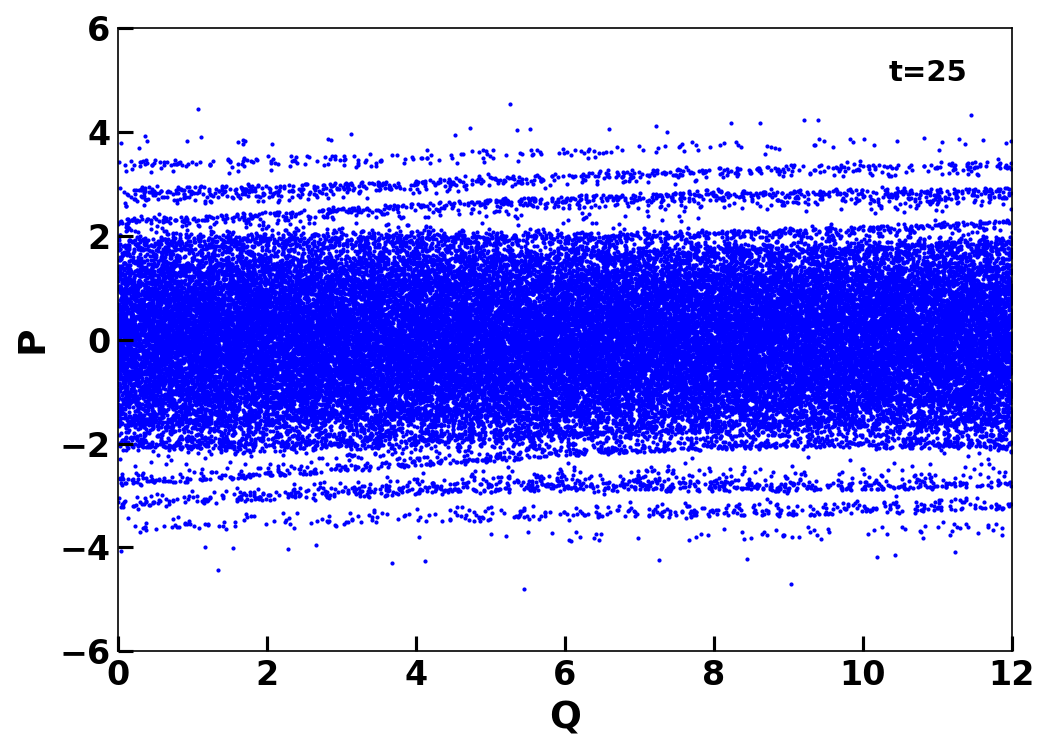}\\
		{(c) SWPIC}
	\end{minipage}
	
\caption{Electric-field amplitude and phase-space comparison for strong Landau damping in PIC and SWPIC. 
Panel~(a) shows the decay of the $E$-field amplitude, with both simulations following the fitted damping rate of $\gamma=-0.236$. Panels~(b) and~(c) show the corresponding phase-space distributions at $t=25$ for PIC and SWPIC, respectively. The PIC system uses $8.8\times 10^4$ while SWPIC uses $10^4$ decorated particles.}
\label{fig:landau_phase_compare}
\end{figure}
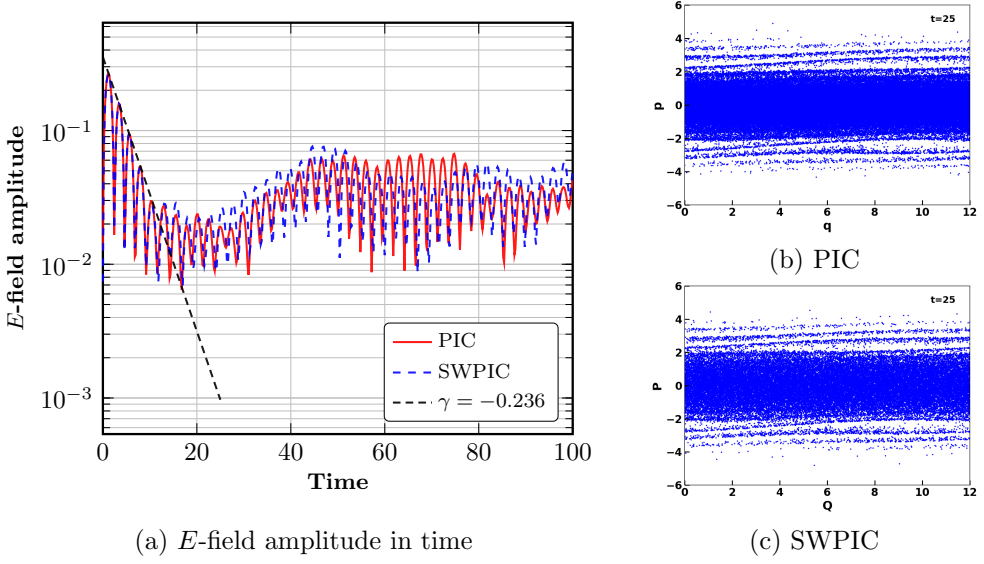

\subsection{Quantitative comparison}
We present numerical benchmarks that compare the SWPIC and PIC schemes. The comparison focuses on three key metrics: error scaling, memory usage, and computational cost. These results demonstrate how introducing internal moment DOFs allows SWPIC to reach the same accuracy with roughly an order of magnitude fewer computational particles. All benchmarks were performed using serial implementations of both methods while leaving the distributed study to a future work. 
Both algorithms are implemented with comparable vectorization.

\paragraph{Error scaling and noise reduction}
A key strength of SWPIC is its ability to reduce statistical noise while preserving accuracy.  
We compare the SWPIC and PIC solutions against a continuous VP reference solution solved by a semi-Lagrangian discontinuous Galerkin scheme~\cite{rossmanith2011positivity, seal2012discontinuous}. 
The relative electric field error is computed by
$\varepsilon_E =
{\|E_{\text{num}} - E_{\text{ref}}\|_{L^2}}/{\|E_{\text{ref}}\|_{L^2}}$,
where $E_{\text{ref}}$ is the reference field. As shown in \autoref{fig:error_vs_particles}, PIC required about $8.8\times 10^4$ particles to reach a 1\% error in $E$, whereas SWPIC achieved comparable accuracy with only $10^4$ decoarted particles. For a more balanced comparison, the same data are plotted versus the total
degrees of freedom in \autoref{fig:error_vs_dof}.  Overall,
\autoref{fig:error_scaling_composite} shows that both PIC and SWPIC display the
expected scaling $\propto N^{-1/2}$, but SWPIC attains a comparable
noise level with significantly fewer computational particles, since each cluster
carries higher--order moments.

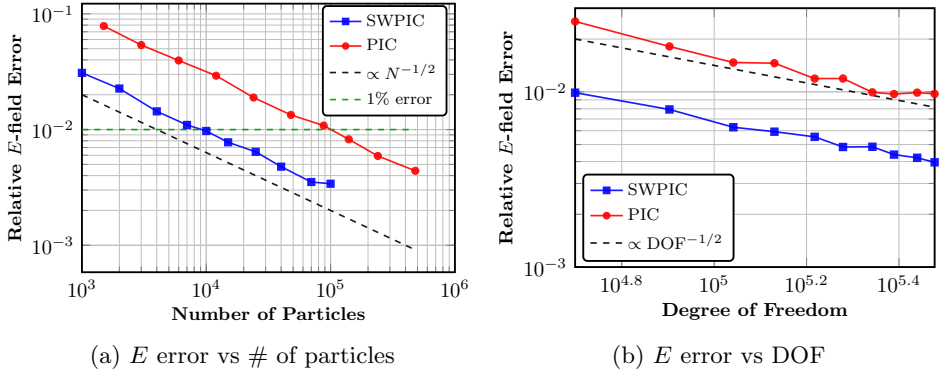
\begin{figure}[!htp]
    \centering
    \begin{minipage}[b]{0.49\textwidth}
\resizebox{\linewidth}{!}{%
    \begin{tikzpicture}
        \begin{axis}[
          scale only axis,
          width=\linewidth,
          height=0.72\linewidth, 
          line width=0.9pt,       
          xmode=log, ymode=log,
          xlabel={Number of Particles},
          ylabel={Relative $E$-field Error},
          xlabel style={font=\bfseries\small, yshift=2pt},
          ylabel style={font=\bfseries\small, yshift=-2pt},
          grid=both,
          legend style={
    at={(0.98,.58)},
    anchor=south east,
    draw=black,
    rounded corners=2pt,
    font=\footnotesize,
    row sep=1pt,
    /tikz/every even column/.style={column sep=3pt},
    legend image post style={xscale=0.9, yscale=0.9}
          },
          legend cell align=left,
          xmin=1e3,xmax=1e6,
          xtick={1e3,1e4,1e5,1e6},
          ytick={1e-3,1e-2,1e-1},
          scaled ticks=false
        ]
        
            \addplot[
                thick,
                blue,
                mark=square*,
                mark size=1.6pt,
                opacity=0.9
            ]
            table[x=N_swpic, y=err_swpic, col sep=comma]
            {Data/error_vs_particles.csv};
            \addlegendentry{SWPIC}

            \addplot[
                thick,
                red,
                mark=*,
                mark size=1.6pt,
                opacity=0.9
            ]
            table[x=N_spic, y=err_spic, col sep=comma]
            {Data/error_vs_particles.csv};
            \addlegendentry{PIC}

            \addplot[
                thick,
                black,
                dashed,
                opacity=0.9,
                domain=1e3:4.8e5,
                samples=50
            ]{0.02*(x/1e3)^(-0.5)};
            \addlegendentry{$\propto N^{-1/2}$}

            \addplot[
                thick,
                green!60!black,
                dashed,
                opacity=0.9,
                domain=1e3:4.8e5
            ]{0.01};
            \addlegendentry{1\% error}

        \end{axis}
    \end{tikzpicture}
}
    \subcaption{$E$ error vs \# of particles}
    \label{fig:error_vs_particles}
    \end{minipage}
    \begin{minipage}[b]{0.46\textwidth}
    \resizebox{\linewidth}{!}{%
    \centering
    \begin{tikzpicture}
        \begin{axis}[
          scale only axis,
          width=\linewidth,
          height=0.72\linewidth, 
          line width=0.9pt,
          xmode=log, ymode=log,
          xlabel={Degree of Freedom},
          ylabel={Relative $E$-field Error},
          xlabel style={font=\bfseries\small, yshift=2pt},
          ylabel style={font=\bfseries\small, yshift=-2pt},
          grid=both,
          legend style={
         at={(0.02,.02)},
        anchor=south west,
        draw=black,
        rounded corners=2pt,
        font=\footnotesize,
        row sep=1pt,
        /tikz/every even column/.style={column sep=3pt},
        legend image post style={xscale=0.9, yscale=0.9}
          },
          legend cell align=left,
            ymin=1e-3, ymax=3e-2,
            enlarge x limits=false,
            enlarge y limits=false,
            ytick={1e-3,1e-2,1e-1},
        ]

            \addplot[
                thick,
                blue,
                mark=square*,
                mark size=1.6pt,
                opacity=0.9
            ]
            table[x=dof_swpic, y=err_swpic, col sep=comma]
            {Data/error_vs_dof.csv};
            \addlegendentry{SWPIC}

            \addplot[
                thick,
                red,
                mark=*,
                mark size=1.6pt,
                opacity=0.9
            ]
            table[x=dof_spic, y=err_spic, col sep=comma]
            {Data/error_vs_dof.csv};
            \addlegendentry{PIC}

            \addplot[
                thick,
                black,
                dashed,
                opacity=0.9,
                domain=5e4:3e5,
                samples=50
            ]{0.02*(x/5e4)^(-0.5)};
            \addlegendentry{$\propto \mathrm{DOF}^{-1/2}$}
        \end{axis}
    \end{tikzpicture}
}
    \subcaption{$E$ error vs DOF}
    \label{fig:error_vs_dof}
    \end{minipage}
    \caption{Convergence of the relative $E$-field error with respect to (a) the number of particles and (b) the total DOFs. A high-resolution solution from a semi-Lagrangian discontinuous Galerkin code is used as the reference baseline.}
    \label{fig:error_scaling_composite}
\end{figure}

\paragraph{Memory usage}
It has been shown that SWPIC achieved the same accuracy as PIC with roughly one order of magnitude fewer computational particles than PIC.  
This reduction in particle count also leads to a smaller memory usage when measured in total DOFs. In PIC, each particle carries three quantities $(Q, P, \psi^*)$, corresponding to $3N$ DOFs for $N$ particles. Each decorated particle carries five quantities $( Q, P,  q^*, p^*, \psi^*)$, or $5N$ DOFs. As shown in \autoref{tab:memory_comparison}, PIC required about $8.8\times10^4$ particles ($2.6\times 10^5$ DOFs), whereas SWPIC reproduced the same result with only $10^4$ decorated particles ($5\times10^4$ DOFs).  
This shows SWPIC uses only about 19\% of the memory required by PIC for comparable accuracy.   

\begin{table}[ht]
\caption{Comparison of total DOFs and memory usage between PIC and SWPIC for comparable accuracy.}
\centering
\begin{tabular}{lcccc}
\hline
Method & Particles & Variables per Particle & DOFs & Relative Memory \\
\hline
PIC     & $8.8\times 10^4$ & 3 ($Q, P, \psi^*$)              & $2.6\times 10^5$ & 100\% \\
SWPIC   & $10^4$ & 5 ($Q, P, q^*, p^*, \psi^*$)  & $5\times 10^4$  & 19\% \\
\hline
\end{tabular}
\label{tab:memory_comparison}
\end{table}

\paragraph{Computational efficiency}
The final study concerns computational speed.  With $10^4$ clusters, SWPIC attains
the same accuracy in about $259\,$s, nearly an order of magnitude faster than
the corresponding PIC run, which required $8.8\times10^4$ particles and
$2{,}378\,$s.  \autoref{fig:runtime_scaling} shows that SWPIC runtimes increase
almost linearly with the number of clusters (slope $\approx 3.0\times10^{-2}$\,s
per cluster).  For very large cluster counts the moment–correction overhead introduces
a mild slowdown, but within the accuracy–equivalent range
$1\text{--}3\times10^4$ clusters, SWPIC remains roughly three to nine times faster
than PIC.  Combined with the $81\%$ reduction in memory footprint, SWPIC is both
efficient and broadly scalable for large kinetic problems.
  
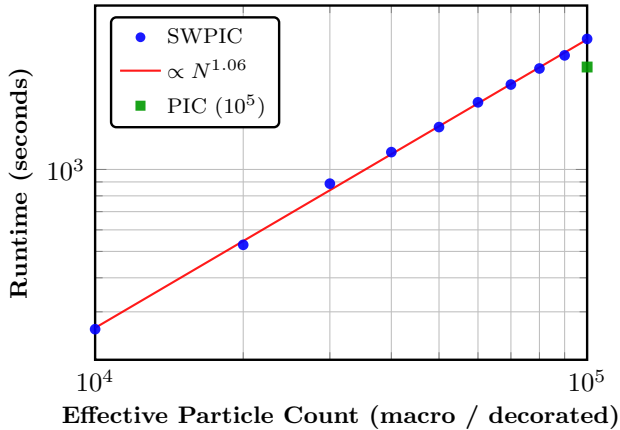
\begin{figure}[!htp]
    \centering
\begin{tikzpicture}
    \begin{axis}[
      scale only axis,
      width=0.5\linewidth,
      height=0.36\linewidth,
      line width=0.9pt,
      xmode=log, ymode=log,
      xlabel={Effective Particle Count (macro / decorated)},
      ylabel={Runtime (seconds)},
      xlabel style={font=\bfseries\small, yshift=2pt},
      ylabel style={font=\bfseries\small, yshift=-2pt},
      grid=both,
      legend style={
        at={(0.03,0.97)},
        anchor=north west,
        draw=black,
        rounded corners=2pt,
        font=\footnotesize,
        row sep=1pt,
        /tikz/every even column/.style={column sep=3pt},
        legend image post style={xscale=0.9, yscale=0.9}
      },
      legend cell align=left,
      xmin=1e4, xmax=1e5,
      ymin=2e2, ymax=4e3,
      xtick={1e4,1e5},
      ytick={1e2,1e3},
      scaled ticks=false
    ]
        \addplot[
            only marks,
            thick,
            blue,
            mark=*,
            mark size=1.6pt,
            opacity=0.9
        ]
        table[x=Clusters, y=SWPIC_Runtime, col sep=comma]
        {Data/timing_data.csv};
        \addlegendentry{SWPIC}

        \addplot[
            thick,
            red,
            solid,
            opacity=0.9
        ]
        table[x=Clusters, y=SWPIC_Fit, col sep=comma]
        {Data/timing_data.csv};
        \addlegendentry{$\propto N^{1.06}$}

        \addplot[
            only marks,
            mark=square*,
            mark size=1.6pt,
            green!60!black,
            opacity=0.9
        ]
        table[x=SPIC_Clusters, y=SPIC_Runtime, col sep=comma]
        {Data/timing_data.csv};
        \addlegendentry{PIC ($10^5$)}
    \end{axis}
\end{tikzpicture}
\caption{Runtime scaling of SWPIC and PIC with respect to the effective number of simulation particles. For SWPIC, this corresponds to the number of decorated particles, while for PIC it corresponds to microscopic particles. The fitted curve indicates near-linear growth (\(\propto N^{1.06}\)) for SWPIC.} 
\label{fig:runtime_scaling}
\end{figure}

\section{Conclusions}
\label{sec:conclusion}
This work investigates a structure-preserving and computationally efficient alternative to conventional PIC methods through the SWPIC formulation, in which each particle carries internal moment degrees of freedom. By revisiting the original Scovel--Weinstein framework, we develop a formulation that exactly inherits the Lie--Poisson structure of the Vlasov--Poisson system, leading to a discrete Hamiltonian and Poisson bracket that ensure non-dissipative energy behavior. Numerical benchmarks for the two-stream instability and strong Landau damping demonstrate that SWPIC reproduces key physical quantities with roughly an order of magnitude fewer particles than standard PIC, thereby reducing statistical noise, memory demands, and runtime while efficiently capturing desired kinetic dynamics.

\appendix

\section{Left trivialization}
\label{app:lefttrivialization}
This appendix records the left trivialization of the cotangent bundle
\(T^*\mathcal{B}\), where 
\(\mathcal{B}=\mathbb{R}^d_Q\times\mathbb{R}^d_P\times\mathbb{R}_\Psi\)
is the \(d\)-dimensional Heisenberg group. We denote elements of $\mathfrak{b}^*$ and $\mathfrak{b}$ as \((q^*,p^*,\psi^*)\in\mathfrak{b}^*\) and $(q,p,\psi)\in\mathfrak{b}$, respectively.

For \(\overline{B}=(\bar{Q},\bar{P},\bar{\Psi})\) and \(B=(Q,P,\Psi)\) in
\(\mathcal{B}\), the group product is
\[
\overline{B} B
=
\big(
Q+\bar{Q},\;
P+\bar{P},\;
\Psi+\bar{\Psi}
+\tfrac12(\bar{Q}\!\cdot\!P - \bar{P}\!\cdot\!Q)
\big).
\]
Thus, the left multiplication map \(L_{\overline{B}}:\mathcal{B}\to\mathcal{B}\) is
\[
L_{\overline{B}}(Q,P,\Psi)
=
(Q+\bar{Q},\;
 P+\bar{P},\;
 \Psi+\bar{\Psi}+\tfrac12(\bar{Q}\!\cdot\!P - \bar{P}\!\cdot\!Q)).
\]
This implies that the tangent map of $L_{\overline{B}}$ at the identity $e=(0,0,0)\in\mathcal{B}$,  is
\begin{align*}
T_eL_{\overline{B}}&:\mathfrak{b}\rightarrow T_{\overline{B}}\mathcal{B}\\
    T_eL_{\overline{B}}(q,p,\psi) &= \bigg(q,p,\psi+\frac{1}{2}(\overline{Q}\cdot p - \overline{P}\cdot q)\bigg)\in T_{\overline{B}}\mathcal{B}.
\end{align*}
Dualizing, we find
\begin{align*}
(T_eL_{\overline{B}})^*&:T^*_{\overline{B}}\mathcal{B}\rightarrow \mathfrak{b}^* \\
    (T_eL_{\overline{B}})^*(Q^*,P^*,\Psi^*) &= \bigg(Q^* - \frac{1}{2}\Psi^*\overline{P},P^* + \frac{1}{2}\Psi^*\overline{Q},\Psi^*\bigg)\in \mathfrak{b}^*.
\end{align*}
The left-trivialization map $\lambda : T^*\mathcal{B}\rightarrow \mathcal{B}\times\mathfrak{b}^*$ is therefore $\lambda(Q,P,\Psi,Q^*,P^*,\Psi^*) =(Q,P,\Psi,(T_eL_{{B}})^*(Q^*,P^*,\Psi^*)) $, or
\begin{align*}
    \lambda(Q,P,\Psi,Q^*,P^*,\Psi^*) = \bigg(Q,P,\Psi,Q^* - \frac{1}{2}\Psi^*{P},P^* + \frac{1}{2}\Psi^*{Q},\Psi^*\bigg).
\end{align*}
The inverse of the left-trivialization map is
\begin{align*}
    \lambda^{-1}(Q,P,\Psi,q^*,p^*,\psi^*)  = \bigg(Q,P,\Psi,q^* + \frac{1}{2}\psi^*P,p^* - \frac{1}{2}\psi^* Q,\psi^*\bigg).
\end{align*}

The canonical Liouville $1$-form on $T^*\mathcal{B}$ is
\begin{align*}
    \theta_{\text{can}} = Q^*\cdot dQ + P^*\cdot dP + \Psi^* d\Psi.
\end{align*}
Pushing forward along $\lambda$ gives the corresponding Liouville $1$-form on the left-trivialized cotangent bundle, $\mathcal{B}\times\mathfrak{b}^*$,
\begin{align}
    \theta = \lambda_*\theta_{\text{can}} =q^*\cdot dQ + p^*\cdot dP+ \psi^*\,d\Psi+ \frac{1}{2}\psi^*(P\cdot dQ-Q\cdot dP).\label{eq:canonical_1form_appendix}
\end{align}
The symplectic form on $\mathcal{B}\times\mathfrak{b}^*$ is then $\omega_{\mathcal{B}\times\mathfrak{b}^*} = - d\theta$.
Computing term by term,
\begin{gather*}
d(q^*\!\cdot dQ)
= dq^*_i \wedge dQ^i,\quad
d(p^*\!\cdot dP) =  dp^{*\,i} \wedge dP_i,\quad 
d(\psi^* d\Psi) = d\psi^* \wedge d\Psi,\\
d\!\Big(\tfrac12\psi^*(P\!\cdot dQ - Q\!\cdot dP)\Big)
=
\tfrac12\,d\psi^*\wedge (P\!\cdot dQ - Q\!\cdot dP)
+ \psi^*\, dP_i\wedge dQ^i,
\end{gather*}
where $i\in\{1,\dots, d\}$.

Introduce unified coordinates
\[
Z^\mu =
\begin{cases}
Q^\mu, & 1\le\mu\le d,\\
P_{\mu-d}, & d<\mu\le2d,
\end{cases}
\qquad
z^*_\mu =
\begin{cases}
q^*_\mu, & 1\le\mu\le d,\\
p^{*\,{\mu-d}}, & d<\mu\le2d.
\end{cases}
\]
Then
\[
- dq^*_i\wedge dQ^i
- dp^{*\,i}\wedge dP_i
=
- dZ^\mu\wedge dz^*_\mu.
\]
Putting all terms together gives the following explicit expression for the symplectic form:
\begin{equation}
\omega_{B\times\mathfrak{b}^*}
= dZ^\nu\wedge dz^*_\nu
+ d\Psi\wedge d\psi^*
+ \psi^*\; dQ^i\wedge dP_i
+ \frac12\; d\psi^*\wedge (Q^i\, dP_i - P_i\, dQ^i).
\label{eq:symplectic_2form_appendix}
\end{equation}

\section{Poisson bracket induced by the left-trivialized symplectic form}
\label{app:poissonbracket}
Starting from the symplectic two–form \eqref{eq:symplectic_2form_appendix}, 
we now derive the corresponding Poisson bracket on
$\mathcal{B}\times\mathfrak{b}^*$.

Let $H\in C^\infty(\mathcal{B}\times\mathfrak{b}^*)$.
The Hamiltonian vector field $X_H$ is defined by
\[
\iota_{X_H}\omega_{B\times\mathfrak{b}^*} = dH.
\]
Write
\[
X_H
=
\Big(\dot Z^\mu\,\frac{\partial}{\partial Z_\mu}
    +\dot z^{*\,\mu}\,\frac{\partial}{\partial z^*_\mu}\Big)
+ \dot\Psi\,\frac{\partial}{\partial\Psi}
+ \dot\psi^*\,\frac{\partial}{\partial\psi^*}.
\]
Contracting \eqref{eq:symplectic_2form_appendix} with $X_H$ and equating to $dH$ yields the Hamilton equations:
\begin{subequations}
\begin{align}
\dot Q_i &= \frac{\partial H}{\partial q^*_i}, &
\dot P_i &= \frac{\partial H}{\partial p^{*}_i},
\\
\dot\Psi &= \frac{\partial H}{\partial\psi^*}
           +\frac12\left(Q^i\frac{\partial H}{\partial p^{*}_i}
                        -P^i\frac{\partial H}{\partial q^{*}_i}\right),
&
\dot\psi^* &= -\frac{\partial H}{\partial\Psi},\\
\dot q^*_i &= -\frac{\partial H}{\partial Q_i}
             -\psi^*\,\frac{\partial H}{\partial p^{*}_i}
             +\frac12\,P_i\,\frac{\partial H}{\partial \Psi},
&
\dot p^{*}_i &= -\frac{\partial H}{\partial P_i}
               +\psi^*\,\frac{\partial H}{\partial q^*_i}
               -\frac12\,Q_i\,\frac{\partial H}{\partial \Psi},
\end{align}
\label{eq:HamApp}
\end{subequations}

For $f,g\in C^\infty(\mathcal{B}\times\mathfrak{b}^*)$, the Poisson bracket is
defined by
\[
\{f,g\}_{\mathcal{B}\times\mathfrak{b}^*}
:= X_f(g)
= \dot\xi^\alpha_f\,\frac{\partial g}{\partial\xi_\alpha},
\]
where $\xi^\alpha$ ranges over 
$(Z^\nu,z^*_\nu,\Psi,\psi^*)$,  
and $\dot\xi^\alpha_f$ is obtained from \eqref{eq:HamApp} with $H=f$.
Substituting these expressions and antisymmetrizing in $f$ and $g$ gives the explicit coordinate form:
\begin{align*}
\{f,g\}_{\mathcal{B}\times\mathfrak b^*}
&= \left(
\frac{\partial f}{\partial Z^\mu}\frac{\partial g}{\partial z_\mu^*}
-
\frac{\partial g}{\partial Z^\mu}\frac{\partial f}{\partial z_\mu^*}
\right)
+ \frac{\partial f}{\partial\Psi}\frac{\partial g}{\partial\psi^*}
- \frac{\partial g}{\partial\Psi}\frac{\partial f}{\partial\psi^*}\\
&\quad- \psi^*\left(
\frac{\partial f}{\partial q_i^*}\frac{\partial g}{\partial p^{*i}}
-
\frac{\partial g}{\partial q_i^*}\frac{\partial f}{\partial p^{*i}}
\right)\\
&\quad
+ \frac12\!\left[
\frac{\partial f}{\partial\Psi}
\!\left(
Q_i\frac{\partial g}{\partial p^{*i}}
-
P_i\frac{\partial g}{\partial q^{*i}}
\right)
-
\frac{\partial g}{\partial\Psi}
\!\left(
Q_i\frac{\partial f}{\partial p^{*i}}
-
P_i\frac{\partial f}{\partial q^{*i}}
\right)
\right].
\end{align*}

\newpage
\bibliographystyle{siamplain}
\bibliography{cumulative_bib_file}
\end{document}